\documentclass[pdflatex,sn-mathphys-num]{sn-jnl}
\usepackage{graphicx,subfig}%
\usepackage{multirow}%
\usepackage{amsthm,amsmath,tikz-cd}%
\usepackage{mathrsfs,cleveref,xfrac}%
\usepackage{xcolor}%
\usepackage{textcomp,physics}%
\usepackage{manyfoot,overpic}%
\usepackage{booktabs}%
\usepackage{algorithm}%
\usepackage{algorithmicx}%
\usepackage{algpseudocode}%
\usepackage{listings, array,tabularx, multirow,multicol}%
\usepackage{enumitem}
\usepackage{hyperref,float}
\usepackage{amsmath, amssymb}
\usepackage{cleveref}
\usepackage{graphicx}
\usepackage{float}
\usepackage{caption}
\usepackage{subfig}
\usepackage{epstopdf,amsmath}

\newcommand{\bu}{{\bf u} }

\newcommand{\Og}{\Omega}

\ifpdf
  \DeclareGraphicsExtensions{.eps,.pdf,.png,.jpg}
\else
  \DeclareGraphicsExtensions{.eps}
\fi
\theoremstyle{thmstyleone}%
\newtheorem{theorem}{Theorem}
\newtheorem{proposition}[theorem]{Proposition}%
\newtheorem{remark}{Remark}%
\newtheorem{lemma}{Lemma}
\begin{document}
\title[A High-Accuracy Symplectic Scheme for a nonlinear transport problem]{A High-Accuracy Symplectic Scheme for a nonlinear transport problem}

\author*[1]{\fnm{Farjana} \sur{Siddiqua}}\email{fas41@pitt.edu}

\author[2]{\fnm{Catalin} \sur{Trenchea}}\email{trenchea@pitt.edu}

\affil*[1,2]{\orgdiv{Mathematics}, \orgname{University of Pittsburgh}, \orgaddress{\street{301 Thackeray Hall}, \city{Pittsburgh}, \postcode{15260}, \state{Pennsylvania}, \country{USA}}}


\abstract{We analyze an 
advection-diffusion-reaction problem with non-homogeneous boundary conditions that models the chromatography process.
We prove stability and error estimates for both constant and affine adsorption, using the
symplectic one-step implicit midpoint method for time discretization and finite elements for spatial discretization. 
In addition, we perform the stability analysis for the nonlinear, explicit adsorption in the continuous and semi-discrete cases. For the nonlinear, explicit adsorption, we also complete the
error analysis for the semi-discrete case and prove the existence of a solution for the fully discrete case. The numerical tests validate our theoretical results.}




\maketitle
\section{Introduction}
The global market for biopharmaceuticals is expanding fast, and $50\%$ of top $100$ drugs will most likely be derived from biotechnology \cite{website1, coker2012biotherapeutics}. The high demand for biopharmaceuticals is due to their effectiveness in treating various illnesses such as diabetes, anemia, cancer, etc. \cite{leader2008protein}. 
For example, monoclonal antibodies \cite{quinteros2017therapeutic} 
are very useful medications 
in treating 
COVID-19 \cite{website2,website3,website4}. 
Other key factors driving the growth of the market are 
the 
rising investments in 
research and development of novel treatments, favorable government regulations, and 
the
increasing adoption of biopharmaceuticals by the global population \cite{website1}. 
To maximize the production capacity while minimizing costs, manufacturers are constantly developing new methods. 
Integrating new technologies into 
existing facilities 
is
more economically viable
than
the alternative of constructing new biomanufacturing facilities, due to financial risks. 
Upstream and downstream processes are typically part of a biomanufacturing facility. In the upstream process, cells cultured by genetically engineered methods release the desired product into a solution, and in the downstream process, the product is purified from the solution \cite{clarke2013bioprocess}. The capacity of production is often limited by downstream purification, usually including chromatography. In the protein chromatography process, when the solution is pushed through the column, the materials in columns separate the proteins \cite{thesis}. 
The 
ideal media for the chromatography columns 
are resin beds, monoliths, and membranes \cite{wang2016development}.
Membrane chromatography \cite{m1,m2,m3} addresses the low efficiency of resin chromatography and uses a porous, absorptive membrane as the packing medium instead of the small resin beads. 
The protein binding capacity is crucial in membrane chromatography as it determines the volume of membrane required for purification. Most adsorption mechanisms, such as ion-exchange membranes, lose the protein binding capacity at relatively low conductivity and often require additional processing stages, causing lower yield and higher production costs. Recent advances in downstream bioprocessing have underscored the need for 
quantitative modeling frameworks that can predict solute transport and adsorption 
behavior within chromatographic membranes and resins. Nonlinear adsorption 
isotherms, such as Langmuir and steric mass-action (SMA) models, capture the 
competitive and capacity-limited binding 
\cite{newref1,newref2,newref3}. The governing advection–diffusion–reaction equations 
considered in this paper are consistent with those derived for membrane chromatography 
columns, where mass transport resistance and nonlinear binding kinetics jointly determine 
the dynamic binding capacity and breakthrough performance \cite{newref4,newref5}. 
Mathematical analyses of these systems are essential for optimizing column design and 
predicting the performance of emerging multimodal membranes used for antibody and 
protein purification 
\cite{newref6,newref7}.
The
recent research in \cite{m2} 
focuses on multimodal membrane-based chromatography. 
The development of a modeling framework capable of characterizing the chromatography process under continuous flow circumstances is critical. 
In this paper, we model this process for creating a simulation tool for transport in a porous medium 
by 
adopting
the reactive transport (advection-diffusion-reaction) problem 
in \cite{thesis}.
Our analysis can be relevant to general nonlinear transport problems, related to variable saturation flow models, e.g. Richards's equation, black oil and multiphase models in poroelastic media \cite{Lott2012,schneid2004priori,MR2114286,MR3489128,MR2652405,RaduPAMM,MR4967716,MuradCorrea2013,MR4766725}.
 \normalcolor
\par
Let $\Omega$ be a bounded domain in $\mathbf{R^d}$, where $d=1,\ 2$, or $3$, see \Cref{fig:domain},
with piecewise smooth boundary $\Gamma$. 
We partition the boundary into three non-overlapping 
parts
$\Gamma=\Gamma_{\text{in}}\cup \Gamma_{\text{n}}\cup \Gamma_{\text{out}}$, 
where the inflow boundary is $\Gamma_{\text{in}}=\{x\in \Gamma:\overrightarrow{n}\cdot \bu(x)<0\}$, the outflow boundary is $\Gamma_{\text{out}}=\{x\in \Gamma:\overrightarrow{n}\cdot \bu(x)>0\}$, 
and
the boundaries comprising no-flow hydraulic zone(s) are  $\Gamma_{\text{n}}={\Gamma} \backslash(\Gamma_{\text{in}}\cup \Gamma_{\text{out}})$. 
Let $\bu$ denote the fluid velocity through the membrane, 
and $\overrightarrow{n}$ denote the unit outward normal to $\Omega$. 
We assume that $\bu$ is given, computed by the Darcy law \cite{darcy}, 
and
satisfying the incompressibility condition $\div\bu=0$, 
and $\bu\cdot \overrightarrow{n}(x,t)=0, \ \ x\in {\Gamma_{\text{n}}},\ t>0$. Let $\omega$ be the total porosity of the membrane ($0\leq \omega \leq 1$), $\rho_s$ be the density of the membrane, ${D}$ be the diffusion tensor that represents the diffusivity of fluid through the membrane, $C$ and $q(C)$ be the concentration in the liquid and absorbed phases respectively. 
For a forcing function $f\in L^2(0,T;L^2(\Omega))$, given velocity $\bu$ and initial concentration $C_0\in L^2(\Omega)$, we consider the following initial boundary value problem of finding 
the concentration $C(x,t)$:
\begin{equation}\label{Eq}
\left\{
\begin{array}{ll}
\omega \partial_t C+(1-\omega)\rho_s\partial_t q(C)+\div{( \bu C)}-\div{(D\grad C)}=f ,\   x\in \Omega,\ t>0 ,
\\ 
C(x,t)=g,  \    x\in {\Gamma_{\text{in}}},\ t>0, 
\\
(D\grad C)\cdot 
\overrightarrow{n}(x,t)=0, \   x\in {\Gamma_{\text{n}}\cup \Gamma_{\text{out}}},\ t>0, 
\\
C(x,0)=C_0(x), \  x\in \Omega.
\end{array}
\right.
\end{equation}
\begin{figure}[h!]
\begin{centering}
\begin{tikzpicture}
    \node at (0,-1.75) (omega) {$\Omega$};
    \node at (0,0.75) (omegaminus) {$\Gamma_{\text{in}}$};
    \node at (0,-4.25) (omegaplus) {$\Gamma_{\text{out}}$};
    \node at (-1.6,-1.75) (omeganleft) {$\Gamma_{\text{n}}$};
    \node at (1.6,-1.75) (omeganright) {$\Gamma_{\text{n}}$};
    \draw (0,0) ellipse (1.25 and 0.5);
    \draw (-1.25,0) -- (-1.25,-3.5);
    \draw (-1.25,-3.5) arc (180:360:1.25 and 0.5);
    \draw [dashed] (-1.25,-3.5) arc (180:360:1.25 and -0.5);
    \draw (1.25,-3.5) -- (1.25,0);
 \end{tikzpicture}
 \caption{
 {Schematic of the computational domain $\Omega$ and the boundary partition used in the formulation.}
}

 \label{fig:domain}
 \end{centering}
\end{figure}
For the inflow boundary, we keep the concentration fixed \cite{bear1988dynamics,darcy}. 
\par
We consider three cases of isotherms: 
(i) constant isotherm, $q(C)=K$, 
(ii) affine isotherm, $q(C)=K_1+K_2C$ and 
(iii) nonlinear, explicit isotherm $q(C)$. 
A typical example of the nonlinear, explicit isotherm is Langmuir's isotherm \cite{m3,lang}, $ q(C)=\frac{q_{max}K_{eq}C}{1+K_{eq}C}$,
where $K_{eq}$ is Langmuir equilibrium constant 
and $q_{max}$ is the maximum binding capacity of the porous medium.
The main result of this paper is 
second-order accurate by using the symplectic one-step implicit midpoint method for the time discretization, at the same computational cost as the Backward Euler method. 
The accuracy comes in two ways, such as the rate of convergence is higher and the mass is better conserved when the midpoint method is used.
The fully discrete formulation of the 
problem 
\eqref{Eq}
is given in \Cref{sec:3}. 
We perform stability analysis and error analysis for the nonlinear 
explicit $q(C)$ in \Cref{sec:4}. 
We 
also 
prove 
the existence of a fully discrete solution, 
and complete 
the
stability analysis and error analysis for the constant and affine $q(C)$ 
in \Cref{sec:4}. 
The numerical tests validating these estimates are given in \Cref{sec:5}.
%
%
\subsection{Previous Work}

The general advection–diffusion equation has been the subject of extensive mathematical study during the past decades \cite{oldwork1,oldwork2,oldwork3,oldwork4,oldwork5,oldwork6}. Recent studies have extended these formulations to model reactive transport with adsorption kinetics relevant to membrane chromatography. In particular, the development of fully implicit stabilized formulations for filtration and separation problems has provided strong theoretical and numerical foundations for analyzing nonlinear adsorption dynamics \cite{wilson2020analysis}. 
Subsequent studies have expanded this framework to incorporate multimodal adsorption isotherms, allowing implicit coupling between the solid‐ and liquid‐phase concentrations and enabling accurate simulation of mixed‐mode adsorption under various operating conditions \cite{wilson2020numerical}.  
Experimental investigations have complemented these modeling efforts through the design and characterization of multimodal membrane adsorbers exhibiting high dynamic binding capacities and strong salt tolerance \cite{wang2015new}, followed by demonstration of effective antibody purification from CHO cell supernatants with high yield and purity using such adsorbers \cite{wang2017antibody}. 
The mathematical and physical basis of these adsorption models was originally established through detailed analysis of coupled advection–diffusion–reaction equations in functionalized membranes \cite{thesis}. 
{Convergent linearization methods for nonlinear transport problems are discussed in \cite{list2016study}. Recent numerical analysis for degenerate porous-media transport models (e.g., Richards’ equation) includes \emph{a posteriori} error estimation and rigorous convergence/error analyses for finite element, DG, and spectral discretizations \cite{mitra2024posteriori,congreve2025error,difonzo2024convergence}.}
The analysis and numerical computations are typically more difficult in the presence of reaction terms, especially nonlinear ones \cite{radu2010analysis}. 
In \cite{thesis}, the author considered constant, linear, and nonlinear adsorption models and analyzed the problem using the first-order accurate backward Euler method for time discretization and the upwind Petrov–Galerkin (SUPG) finite element method for spatial discretization, with numerical validation for each of the a priori estimates. On the numerical side, the SUPG and related stabilization techniques \cite{brooks1982streamline,burman2010consistent,john2011error} remain the primary tools for advection-dominated transport, while structure-preserving symplectic schemes \cite{hairer2006geometric,sanz1992symplectic} motivate the second-order midpoint approach used in this paper. A high-order linearly implicit scheme for reaction–diffusion equations was proposed in \cite{wilson2018numerical}, offering improved temporal accuracy and stability for stiff reactive systems. This approach highlights the importance of efficient time-integration strategies, which also motivates the symplectic midpoint formulation used in the present work. 
Building upon these prior advances \cite{wilson2020analysis,wilson2020numerical,wilson2018numerical,wang2015new,wang2017antibody}, the present work extends the stabilized finite‐element framework by incorporating a symplectic implicit midpoint time‐integration scheme to improve temporal accuracy, ensure mass conservation, and enhance numerical stability for advection‐dominated reactive transport problem.

\section{Notation and Preliminaries}\label{sec:2}
We denote the $L^2(\Omega)$ norm and inner product by $\|\cdot\|$ and, $(\cdot,\cdot)$ respectively. We denote the usual Sobolev spaces $W^{m,p}(\Omega)$ with the associated norms $\|\cdot\|_{W^{m,p}(\Omega)}$ and in the case when $p=2$, we denote $W^{m,2}(\Omega)=H^m(\Omega)=\{v\in L^2(\Omega):\frac{\partial^{\alpha}v}{\partial  x ^{\alpha}}\in L^2(\Omega), \ |\alpha |\leq r\}$ where $\alpha$ is a multi-index, with norm $\|v\|_r=\Bigg(\sum_{|\alpha|\leq r}\int_{\Omega}\Big|\frac{\partial^{\alpha}v}{\partial  x ^{\alpha}}\Big|^2d \Omega\Bigg)^{1/2}.$ We denote minimum eigenvalue of $D$ as $\lambda$. The function space for the liquid phase concentration is defined as:
\begin{align*}
    &H_{0,\Gamma_{\text{in}}}^1(\Omega):=\{v:v\in H^1(\Omega)\ \text{with} \ v=0 \ \text{on}\ \Gamma_{\text{in}}\}.
\end{align*}
Let $X_{0,\Gamma_{\text{in}}}\subset H_{0,\Gamma_{\text{in}}}^1(\Omega)$. We define the space $H^{1/2}(\Gamma_{\text{in}}):=\{g\in L^2(\Gamma_{\text{in}}): \|g\|_{H^{1/2}(\Gamma_{\text{in}})}<\infty \}$ where $$\|g\|_{H^{1/2}(\Gamma_{\text{in}})}=
\inf_{ G\in H^1(\Omega)\atop
G\big|_{\Gamma_{\text{in}}=g}}\|G\|_{H^1(\Omega)}.$$
The Bochner space \cite{adams2003sobolev} norms are $$\|C\|_{L^2(0,T;X)}=\Bigg(\int_0^T\|C(\cdot,t)\|_{X}^2 dt\Bigg)^{\frac{1}{2}},\ \|C\|_{L^{\infty}(0,T;X)}=\text{ess}\sup_{0\leq t \leq T}\|C(\cdot,t)\|_{X}.$$
We also define 
the
discrete $L^p$-norms with $p=2\ \text{or}\ \infty$
$$\|C\|_{L^2(0,T,X)}=\bigg(\Delta t\sum_{n=0}^N\| C^{n}\|_{X}^2\bigg),
\quad
\|C\|_{L^{\infty}(0,T,X)}= \max_{0\leq n \leq N}\|C^n\|_X.$$
For the Finite Element approximation, we consider a regular triangulation of $\Omega$, $\mathcal{T}_h=\{A\}$ with $\Omega=\bigcup_{A\in \mathcal{T}_h}A$.
We choose a finite dimensional subspace $X^h\subset H^1(\Omega)$ and define 
$$X_{0,\Gamma_{\text{in}}}^h=\{v_h\in X^h:\ v_h=0 \ \text{on} \ \Gamma_{\text{in}}\}$$ with $\Omega$ a polyhedron, $X_{0,\Gamma_{\text{in}}}^h\subset H_{0,\Gamma_{\text{in}}}^1(\Omega)$. Let $X^{*}$ be the dual space of $X_{0,\Gamma_{\text{in}}}$, with norm $\|f\|_{*}=\sup_{v\in X_{0,\Gamma_{\text{in}}}}\frac{(f,v)}{\|\grad v\|}$. We denote $X_{\Gamma_{\text{in}}}^h$ as the restriction of functions in $X^h$ to the boundary $\Gamma_{\text{in}}$ and define 
$X_0^h=\{v_h\in X^h:\ v_h=0 \ on \ \partial\Omega\}$ with $\Omega$ a polyhedron, $X_0^h\subset H_0^1(\Omega)$.
Throughout, $K$ will denote a constant taking different values in different instances. We assume that there exists a $k\geq 1$ such that $X^h$ possesses the approximation property, 
\begin{align}\label{inq1}
\inf_{C_h\in X^h}\|C-C_h\|_s\leq K h^{r-s}\|C\|_r,\ \text{for} \ \ s=0,1 \ \ \text{and} \ \ 1\leq r \leq k+1.
\end{align}
For example, \eqref{inq1} holds if $X^h$ consists of piecewise polynomials of degree $\leq k$. We assume that a similar approximation holds on $X_0^h$. In particular, if $C\in H^r(\Omega)\cap H_0^1(\Omega)$, we will use 
\begin{align}\label{inq2}
\inf_{C_h\in X_0^h}\|C-C_h\|_1\leq K h^{r-1}\|C\|_r.
\end{align}
We further assume that the space $X_{\Gamma_{\text{in}}}^h$ possesses the approximation property 
\begin{align}\label{inq3}
\inf_{C_h\in X_{\Gamma_{\text{in}}}^h}\|C-C_h\|_{0,\Gamma_{\text{in}}}\leq K h^{r-1/2}\|C\|_{r-1/2,\Gamma_{\text{in}}}.
\end{align}
\begin{lemma}\label{poin2}
For all $v\in H_{0,\Gamma_{\text{in}}}^1(\Omega)$, there exists a constant ${\Tilde{K}}_{PF}$ such that
\begin{align*}
\|v\|_1 \leq {\Tilde{K}}_{PF}\|\grad v\|.
\end{align*}
\end{lemma}
\begin{proof}
This is the direct consequence of the Poincar\'e inequality that holds for $v\in H_{0,\Gamma_{\text{in}}}^1(\Omega)$ \cite{ervin2015limiting}.
\end{proof}

\begin{lemma}\label{dop}(See \cite[p.154]{layton2002connection})
Let $\mathcal{P}$ and $\mathcal{P}^1$ be the orthogonal projections with respect to the $L^2$ inner product $(u,v)$ and $H^1$ inner product $(\grad u,\grad v)$, respectively. Then, for any $w\in X$,
$$\grad\mathcal{P}_{X^h}w=\mathcal{P}^1_{\grad X^h}\grad w.$$
\end{lemma}
\begin{lemma}\label{interpolant}
Given $g\in H^{r-1/2}(\Gamma_{\text{in}})$ for $r\geq 1$, let $\Pi_h g$ denote the $X_{\Gamma_{\text{in}}}^h$-interpolant of $g$. Then, if $X^h$ satisfies the approximation properties \eqref{inq1}-\eqref{inq3},
\begin{align}\label{inq4}
\inf_{{\hat{C}}_h\in X^h\atop {\hat{C}}_h|_{\Gamma_{\text{in}}} =\Pi_hg} \|C-{\hat{C}}_h\|_1\leq K h^{r-1}\|C\|_{r}.
\end{align}
\end{lemma}
\begin{proof} This proof follows the proof of \cite[Lemma 4]{fix1983finite},
given here 
for the reader's convenience.
Let $\Pi_h C$ denote $X^h$-interpolant of $C$ and $\Pi_h g$ denote $X_{\Gamma_{\text{in}}}^h$- interpolant of $g$. Then, for ${\hat{C}}_h|_{\Gamma_{\text{in}}}=\Pi_h g$, we write the triangle inequality
\begin{equation}\label{inq5}
    \|C-{\hat{C}}_h\|_1\leq \|C-\Pi_hC\|_1+\|{\hat{C}}_h-\Pi_hC\|_1. 
\end{equation}
From the interpolation theory \cite{brenner2008mathematical}
we get
\begin{equation}\label{inq6}
    \|C-\Pi_hC\|_1\leq K h^{r-1}\|C\|_r. 
\end{equation}
We may choose ${\hat{C}}_h$ so that it has the same value at all interior nodes as does $\Pi_h C$. Since ${\hat{C}}_h|_{\Gamma_{\text{in}}}=\Pi_h g$ and $(\Pi_hC)\Big|_{\Gamma_{\text{in}}}=\Pi_h g$, we 
obtain
$({\hat{C}}_h-\Pi_hC)=0$, which concludes the argument.
\end{proof}
\subsection{Assumptions and preliminary results}
We 
use 
the following 
subset of 
assumptions 
considered in
\cite{thesis}:
\begin{enumerate}[label=\text{(F\arabic*)},ref=(F\arabic*)]
\item \label{F1}
$\omega$ and $\rho_s$ are constants in time and space \cite{darcy}. 
\item 
\label{F2}
$\bu$ is nonzero and bounded in $L^{\infty}$ norm \cite{chrispell2007fractional, schneid2004priori}.
\item 
\label{F3}
$D(x)=[d_{ij}]_{i,j=1,2,\cdots,n}$ is symmetric positive definite and  $\|D\|_{\infty}\leq \beta_1,\ |\frac{\partial}{\partial x_i}d_{ij}|\leq \beta_2$, for all $i,j$ \cite{arbogast1995characteristics,darcy,chrispell2007fractional,schneid2004priori}.
\item 
\label{F4}
There exists a unique solution $C\in L^{\infty}(0,T,L^2(\Omega))\cap L^{2}(0,T,H^1(\Omega))$ \cite{schneid2004priori}.
\item 
\label{F5}
$q=q(C)\in C^1$ is an explicit, 
Lipschitz continuous function of $C$, $q(0)=0$, $q(C)>0$ for $C>0$ and $q(C)$ is strictly increasing. Moreover, we assume that $q'(C)\geq \kappa_1>0 \ \forall C\geq0$
\cite{schneid2004priori,darcy,radu2010analysis,radu2011mixed,barrett1997finite,dawson1998analysis}.
\item 
\label{F6}
The rate of increase in adsorption is Lipschitz continuous and bounded above so that $\frac{dq}{dC}=q'(C)\leq \kappa_2$ \cite{darcy}.
\item 
\label{F7}
The second derivative of the adsorption, $q''(C)$, is Lipschitz continuous and bounded. 
\end{enumerate}
\begin{remark}
In our analysis we drop the assumption “$C(x,t)$ is nondecreasing in time at every $x$ and $C(x,t)=0$ on $\Gamma_{\text{in}}$" stated in \cite{thesis}. Instead, we consider 
the non-homogeneous boundary condition at the inflow boundary.
\end{remark}
In \cite{thesis,barrett1997finite,darcy,schneid2004priori,dawson1998analysis}, another assumption on the continuous and the discrete solution was imposed, namely that ``$C$ is non-negative”. 
Using a maximum principle argument, we now prove that the continuous solution is positive and bounded above for all $(x,t)\in \Omega\times (0,T)$.

\begin{proposition}\label{prop:pr}
%
%
Assuming  no forcing term $f=0$ and the positivity of the initial condition $0<C_0(x)$, 
we have that 
\begin{align*}
0 < C(x,t) \leq \max_{x\in \Omega} \{ C(x,0) , g(x) \}
\qquad \mbox{for all }(x,t)\in \Omega\times [0,T).
\end{align*} 
\end{proposition}
\begin{proof}
From the incompressibility condition, we have
\begin{align*}
    \div{( \bu C)}=(\div{ \bu})C+ \bu\cdot \grad{C}= \bu \cdot \grad{C},
\end{align*}
If we rewrite the adsorption term as
\begin{align*}\label{v3}
\frac{\partial q}{\partial t}=\frac{\partial q}{\partial C}\frac{\partial C}{\partial t}=q'(C)\frac{\partial C}{\partial t},
\end{align*}
then the equation \eqref{Eq} becomes
\begin{equation}\label{maxprinciple}
   (\omega+(1-\omega)\rho_sq'(C)) \partial_t C+\bu\cdot \grad C-\div{(D\grad C)}=f,\ x\in \Omega,\ t>0. 
\end{equation}
Since $q'(C)>0$ by assumption \ref{F5}, we can divide \eqref{maxprinciple} by $(\omega+(1-\omega)\rho_sq'(C))$. 
Hence, assuming $f=0$, \eqref{maxprinciple} writes,
\begin{align*}
   & -\partial_t C+\sum_{i,j=1}^n \bigg((\omega+(1-\omega)\rho_sq'(C))^{-1}D_{ij}\bigg)\frac{\partial^2 C}{\partial x_i \partial x_j}
    \\&
    +\sum_{j=1}^n\Bigg((\omega+(1-\omega)\rho_sq'(C))^{-1}\bigg(\frac{\partial D_{ij}}{\partial x_i}-u_j\bigg)\Bigg)\frac{\partial C}{\partial x_j}=0.
\end{align*}
Suppose the claim in the proposition is false. Then there is a $y\in \Bar{\Omega}$ and $T>0$ such that $C(y,T)=0$ and $C(x,t)>0$ for $(x,t)\in\Bar{\Omega}\ \cross [0,T)$. Therefore, by the Maximum Principle \cite[pages 173-174]{protter2012maximum}, the maximum of $C(x,t)$ is on the boundary and $(D\grad C)\cdot 
\overrightarrow{n}(x,t)<0$. This contradicts the boundary condition 
in \eqref{Eq}, 
which concludes the argument. 
\end{proof}
\section{Variational Formulation}\label{sec:3}
Let $\mathcal{P}$ be the orthogonal projection on $H^1(\Omega)$ with respect to the $L^2$ inner product $(u,v)$. The standard Galerkin variational formulation for the transport problem \eqref{Eq} is: 
Find $C \in H^1(\Omega)$ such that ${C  \Big{|}}_{\Gamma_{\text{in}}}=g$ and 
\begin{align}
\label{variational1new}
    \bigg(\frac{\partial}{\partial t}(\omega C+(1-\omega)\rho_s\mathcal{P}(q(C))),v\bigg)
    +(  \bu \cdot \grad{C},v)+&(D\grad{C},\grad{v})
    =( f,v),
\end{align}
for all $v\in H_{0,\Gamma_{\text{in}}}^1(\Omega)$.
%
%
Next, we write a finite element approximation for \eqref{Eq}.
\subsection{Semi-Discrete in Space Approximation}
Let $\mathcal{P}_h$ be the orthogonal projection on $X^h(\Omega)$ with respect to the $L^2$ inner product $(u,v)$, 
and 
$g_h=\Pi_h g$ 
an interpolant of $g$. Then we obtain the following semi-discrete in-space formulation:
Find $C_h\in X^h$ such that ${C_h \Big{|}}_{\Gamma_{\text{in}}}=g_h$ and $\forall\ v_h\in X_{0,\Gamma_{\text{in}}}^h(\Omega),$
\begin{align}
\label{variational2new}
    \bigg(\frac{\partial}{\partial t}(\omega C_h+(1-\omega)\rho_s\mathcal{P}_h(q(C_h))),v_h\bigg)
    + (  \bu \cdot \grad{C_h},v_h)
    + (D\grad{C_h},\grad{v_h})
    =( f,v_h).
\end{align}

\subsection{Fully-discrete approximation}\label{subsec:32}
We partition the time interval as $t_0=0<t_1<t_2<\cdots<t_N=T$. Let $\Delta t=t_{n+1}-t_n$ be the uniform time step size, $t_n=n\Delta t$, $t_{n+1/2}=\frac{t_n+t_{n+1}}{2}$, and $f^n(x)=f(x,t_n)$. We also denote by $C^n_h( x)$ the Finite Element approximation to $C( x,t_n)$. 
The midpoint method for time discretization in Finite Element Approximation: Given $C^n_h\in X^h$, find $C^{n+1}_h\in X^h$ such that ${C_h^{n+1} \Big{|}}_{\Gamma_{\text{in}}}=g_h$ satisfying
\begin{equation}\label{nempc}
\begin{aligned}
 &\bigg(\omega\frac{C^{n+1}_h-C^n_h}{\Delta t}+(1-\omega)\rho_s \frac{q(C^{n+1}_h)-q(C^n_h)}{\Delta t},v_h\bigg)+(  u \cdot \grad{C^{n+1/2}_h},v_h)
 \\&+(D\grad{C^{n+1/2}_h},\grad{v_h})
 =( f^{n+1/2},v_h), \ \forall v_h\in X_{0,\Gamma_{\text{in}}}^h(\Omega),
\end{aligned}
\end{equation}
where $C_h^{n+\frac{1}{2}}$ denotes $\frac{C_h^n+C_h^{n+1}}{2}$.
Equivalently, $\forall v_h\in X_{0,\Gamma_{\text{in}}}^h(\Omega)$,
\begin{equation}\label{nempcn}
\begin{aligned}
 &\bigg((\omega+(1-\omega)\rho_sq'(C_h^{n+1/2}))\frac{C^{n+1}_h-C^n_h}{\Delta t},v_h\bigg)+(  u \cdot \grad{C^{n+1/2}_h},v_h)+(D\grad{C^{n+1/2}_h},\grad{v_h})
 \\&=( f^{n+1/2},v_h).
\end{aligned}
\end{equation}
To simplify computation, we use the refactorization of the midpoint method \cite{burkardt2020refactorization} for time discretization.
Given $C^n_h\in X^h$, find $C^{n+1}_h\in X^h$ such that ${C_h^{n+1} \Big{|}}_{\Gamma_{\text{in}}}=g_h$ satisfying\\
Step 1: Backward Euler method approximating \eqref{variational2new} on time interval $[t_n,t_{n+1/2}]$, $\forall\ v_h\in X_{0,\Gamma_{\text{in}}}^h(\Omega)$,
\begin{align}
&\bigg( (\omega+(1-\omega)\rho_s)\frac{q(C^{n+1/2}_h)-q(C^n_h)}{\Delta t/2},v_h\bigg)
+(  \bu \cdot \grad{C^{n+1/2}_h},v_h)+(D\grad{C^{n+1/2}_h},\grad{v_h})
\notag
\\
&
\label{mp1}
=( f^{n+1/2},v_h).
\end{align}
\\
Step 2: Forward Euler method on time interval $[t_{n+1/2},t_{n+1}]$, $\forall \ v_h\in X_{0,\Gamma_{\text{in}}}^h(\Omega)$
\begin{align}
&
\bigg( (\omega+(1-\omega)\rho_s)\frac{q(C^{n+1}_h)-q(C^{n+1/2}_h)}{\Delta t/2},v_h\bigg)
+( \bu \cdot \grad{C^{n+1/2}_h},v_h)+(D\grad{C^{n+1/2}_h},\grad{v_h})
\notag
\\
\label{mp2}
&
=( f^{n+1/2},v_h).
\end{align}
\begin{remark}
Step 2 is equivalent to a linear extrapolation $C_h^{n+1}=2C_h^{n+1/2}-C_h^n$. 
\end{remark}
\subsection{Time-integrated finite element formulation}
For the error analysis 
of the case of a nonlinear, explicit adsorption, we 
use a time-integrated version of the transport equation introduced in \cite{nochetto1988approximation} and applied in different formulations \cite{arbogast1996nonlinear,woodward2000analysis,schneid2004priori}. To develop the time-integrated finite element discretization, we rewrite \eqref{Eq} by integrating in time to obtain
\begin{equation}\label{timeinteg}
    \omega C+(1-\omega)\rho_s q(C)+\int_0^t\bu\cdot\grad C dt'-\div\int_0^tD\grad Cdt'=\int_0^tf dt'+\omega C_0+(1-\omega)q(C_0).
\end{equation}
Testing \eqref{timeinteg} by $v\in H_{0,\Gamma_{in}}^1(\Omega)$ 
we get,
\begin{equation}\label{timeintegv}
    \begin{aligned}
        &(\omega C,v)+((1-\omega)\rho_s q(C),v)+\bigg(\int_0^t\bu\cdot\grad C dt',v\bigg)-\bigg(\div\int_0^tD\grad Cdt',v\bigg)
        \\&=\bigg(\int_0^tf dt',v\bigg)+(\omega C_0,v)+((1-\omega)q(C_0),v).
    \end{aligned}
\end{equation}
Then the semi-discrete in space variational formulation is
: Find $C_h\in X^h$ such that $C_h\big|_{\Gamma_{in}}=g_h$ and
\begin{align}
\label{timeintegfemsemi}
&
(\omega C_h,v_h)+((1-\omega)\rho_s q(C_h),v_h)+\bigg(\int_0^t\bu\cdot\grad C_h dt',v_h\bigg)-\bigg(\div\int_0^tD\grad C_hdt',v_h\bigg)
\\
&
=\bigg(\int_0^tf dt',v_h\bigg)+(\omega C_0,v_h)+((1-\omega)q(C_0),v_h),\quad \forall v_h\in X_{0,\Gamma_{in}}^h(\Omega).
\notag
\end{align}
Next, the fully discrete variational formulation using the midpoint time discretization can be written as: Given $C_h^n\in X^h$, find $C_h^{n+1}\in X^h$ such that $C_h^{n+1}\big|_{\Gamma_{in}}=g_h$ and
\begin{align}
&
\bigg( \omega C_h^{N+1} 
+( (1-\omega)\rho_s q(C_h^{N+1}),v_h \bigg)
+ \bigg(\sum_{n=0}^{N}\bu\cdot\grad C_h^{n+1/2} ,v_h\bigg)-\bigg(\div\sum_{n=0}^{N}D\grad C_h^{n+1/2},v_h\bigg)
\notag
\\
\label{timeintegfemfull}
&
= \bigg(\sum_{n=0}^{N}f^{n+1/2} ,v_h\bigg)+(\omega C_0,v_h)+((1-\omega)q(C_0),v_h),\quad \forall v_h\in X_{0,\Gamma_{in}}^h(\Omega).
\end{align}
\section{Time-Dependent Analysis}\label{sec:4} 
In this section, we first construct $\hat{C}$,  a continuous extension of the Dirichlet data $g$ inside the domain $\Omega$, to deal with the non-homogeneous boundary condition. 
Then we perform the stability and error analysis for the time-dependent problem. Detailed proofs of all theorems can be found in \cite{siddiqua2024spurious}.

\subsection{Construction of $\hat{C}$}
Denote $\hat{C}$ as the solution of the following elliptic problem with nonhomogeneous mixed boundary conditions:
\begin{align}\label{construction}
-\div{(D\grad \hat{C})}+\hat{C}&=0,\  x\in \Omega,
\\
\hat{C}
&={g},\ \text{if $x\in \Gamma_{\text{in}}$},\notag
\\
({D}\grad \hat{C})\cdot \overrightarrow{n}
&=0,\ \text{if $x\in \Gamma_{\text{n}}\cup \Gamma_{\text{out}}$}.
\notag
\end{align}
\begin{lemma}\label{Eu}
For every $f\in L^2(\Omega)$ and every $g\in H^{1/2}(\Gamma_{\text{in}})$, there exists a unique solution $\hat{C}\in H^2(\Omega)$ of \eqref{construction} under the compatibility condition $D\grad{g}\cdot \overrightarrow{n}=0$ if $x\in \Gamma_{\text{in}}\cap \Gamma_n$ 
such that
\begin{align}\label{b1}
 \|\hat{C}\|^2\leq 4(K\beta_1)^2\|g\|_{L^2(\Gamma_{\text{in}})}^2,
 \qquad
 \|\grad \hat{C}\|^2\leq \frac{2(K\beta_1)^2}{\lambda}\|g\|_{L^2(\Gamma_{\text{in}})}^2.
 \end{align}
\end{lemma}
\begin{remark}
The existence and uniqueness proof for the more general case of \Cref{Eu} can be found in \cite[Theorem 2.4.2.7]{grisvard2011elliptic}. 
 \end{remark}
 \begin{lemma}
 \label{chath}
 Let the domain $\Omega$ be a convex polyhedral. Given $g^h\in X_{\Gamma_{\text{in}}}^h$, there exists a $\hat{C^h}\in X^h$ such that $\hat{C^h}|_{\Gamma_{\text{in}}}=g^h$ and $\|\hat{C^h}\|_{H^1(\Omega)}\leq K \|g^h\|_{H^{1/2}(\Gamma_{\text{in}})}$.
 \end{lemma}
 \begin{proof}
 When $\Omega$ is two-dimensional, we use a similar technique to 
 \cite{gunzburger1992treating}. 
 Under the compatibility condition ${D}\grad {g^h}\cdot \overrightarrow{n}=0,\ \text{when} \ x\in \Gamma_{\text{in}}\cap \Gamma_{\text{n}}$, let $\hat{C}\in H^1(\Omega)$ be the solution of
\begin{align}\label{construction1}
-\div{(D\grad \hat{C})}+\hat{C}&=0,\  x\in \Omega,
\\
\hat{C}
&={g^h},\ \text{when $x\in \Gamma_{\text{in}}$},\notag
\\
({D}\grad \hat{C})\cdot \overrightarrow{n}
&=0,\ \text{if $x\in \Gamma_{\text{n}}\cup \Gamma_{\text{out}}$}.\notag
\end{align}
Since $X^h$ is assumed to be a continuous finite element subspace, we see that $g^h$ is continuous and piecewise smooth along the boundary $\Gamma_{\text{in}}$, so that $g^h\in H^{1/2+\epsilon}(\Gamma_{\text{in}})$ for $0<\epsilon\leq \frac{1}{2}$. Thus, by elliptic regularity, we derive that $\hat{C}\in H^{1+\epsilon}(\Omega)$ and $\|\hat{C}\|_{1+\epsilon}\leq K\|g^h\|_{1/2+\epsilon,\Gamma_{\text{in}}}$ for $0<\epsilon\leq \frac{1}{2}$.
Let $\hat{C^h}:=\Pi_h\hat{C}$ be the $X^h$-interpolant of $\hat{C}$ so that $\hat{C^h}|_{\Gamma_{\text{in}}}=g^h$. Then, we have the estimates $\|\hat{C}-\Pi_h\hat{C}\|_1\leq K h^{\epsilon}\|\hat{C}\|_{1+\epsilon}$ which can be proven as in, e.g., \cite{dupont1980polynomial}. Thus, we get 
\begin{align*}
    \|\hat{C^h}\|_1=\|\Pi_h\hat{C}\|_1&\leq \|\hat{C}-\Pi_h\hat{C}\|_1+\|\hat{C}\|_1
    \leq K(h^{\epsilon}\|\hat{C}\|_{1+\varepsilon}+\|\hat{C}\|_1)
    \leq K\|g^h\|_{1/2,\Gamma_{\text{in}}},
\end{align*}
where in the last step we used an inverse assumption on $X^h_{\Gamma_{\text{in}}}$: there exists a constant $K$, independent of $h$, $p^h$ such that 
$$\|p^h\|_{s,\Gamma{\text{in}}}\leq K h^{t-s}\|p^h\|_{t,\Gamma_{\text{in}}},\ \forall p^h\in X^h_{\Gamma_{\text{in}}},\ 0\leq t\leq s\leq 1.$$
 Since the usual interpolant 
 used in the two-dimensional case is not defined in three dimensions for $H^r(\Omega)$-functions, $r\leq \frac{3}{2}$, we use the Scott-Zhang interpolant \cite{dauge2006elliptic} when $\Omega$ is three-dimensional. The Scott-Zhang interpolant is well-defined for any function in $H^1(\Omega)$ \cite{scott1990finite}.
 \end{proof}
 \subsection{Nonlinear, Explicit Isotherm}
 In this subsection, we start with the numerical analysis for the nonlinear, explicit isotherm. We consider the variational formulation 
 \eqref{variational1new}, 
 the semi-discrete in-space formulation 
 \eqref{variational2new}, and the fully discrete formulation 
 given in  \Cref{subsec:32}. 
 First, 
we show a total mass balance 
relation for this 
nonlinear explicit
isotherm. 
We denote the antiderivative of the isotherm by
$Q(\alpha)=\int_{0}^{\alpha} q(s) ds$,
and
\begin{align*}
\mathcal{E}(t)
&
=
\frac{3\omega}{4}\int_0^t\bigg\|D^{1/2}\grad C-\frac{8}{3}D^{-1/2}{\hat{C}}\bu\bigg\|^2\ dr
\\
& \quad 
+ \frac{3\omega}{4}\int_0^t\bigg\|D^{1/2}\grad C-\frac{8\rho_s q(\hat{C})(1-\omega)}{3\omega}D^{-1/2}\bu\bigg\|^2\ dr
\\
&
\quad
+ \frac{3\omega}{4}\int_0^t\bigg\|D^{1/2}\grad C-\frac{8}{3}D^{-1/2}\grad \hat{C}\bigg\|^2 dr
\\
&
\quad
+ \frac{3\omega}{4}\int_0^t\bigg\|D^{1/2}\grad C-\frac{8(1-\omega)\rho_s q'(\hat{C})}{3\omega}D^{-1/2}\grad \hat{C}\bigg\|^2 dr
\\
&
\quad
+ 
\|\omega C(t)+(1-\omega)\rho_s q(C(t))-2(\omega \hat{C}+(1-\omega)\rho_s q(\hat{C}))\|^2,
\end{align*}
also
\begin{align*}
\mathcal{B}(t)
&
=
\frac{3\omega}{4} \int_0^t \! \bigg\|\frac{8}{3}D^{-1/2}{\hat{C}}\bu\bigg\|^2\ dr
+ \frac{3\omega}{4}\int_0^t\bigg\|\frac{8\rho_s q(\hat{C})(1-\omega)}{3\omega}D^{-1/2}\bu\bigg\|^2\ dr
\\
& \quad
+
\frac{3\omega}{4}\int_0^t\bigg\|\frac{8}{3}D^{-1/2}\grad \hat{C}\bigg\|^2\ dr
+ \frac{3\omega}{4}\int_0^t\bigg\|\frac{8(1-\omega)\rho_s q'(\hat{C})}{3\omega}D^{-1/2}\grad \hat{C}\bigg\|^2 dr
\\
&
\quad
+ \|\omega C(0)+(1-\omega)\rho_s q(C(0))-2(\omega \hat{C}+(1-\omega)\rho_s q(\hat{C}))\|^2.
\end{align*}
\begin{theorem}
\label{thm:neth2neweq}
Assume that \ref{F1}-\ref{F6} are satisfied, 
$f\in L^2(0,T;L^2(\Omega))$, 
the variational 
problem 
\eqref{variational1new} has a solution $C\in L^{\infty}(0,T,L^2(\Omega))\cap L^{2}(0,T,H^1(\Omega))$, 
and 
let ${\hat{C}}$ be solution of (\ref{construction}).
Then the following total mass balance relation holds:
\begin{align}
\label{eq:stability-continuous}
&
\|\omega C(t)+(1-\omega)\rho_s q(C(t))\|^2+\omega\int_0^t \|D^{1/2}\grad C(r)\|^2\, dr
\\
& 
\qquad
+ 4(1-\omega)\rho_s\int_0^t 
\Bigg( 
\int_{\Omega}q'(C(r))(D^{1/2}\grad C(r))^2\, d\Omega
+ \int_{\Gamma_{\text{out}}}   Q(C(r))(\bu\cdot \overrightarrow{n})ds
\Bigg)\, dr
\notag
\\
&
\qquad
+ 2\omega\int_0^t \Bigg(\int_{\Gamma_{\text{out}}}   C^2(\bu\cdot \overrightarrow{n})ds\Bigg)\, dr
+\mathcal{E}(t)
\notag
\\
&
= \|\omega C_0+(1-\omega)\rho_s q(C_0)\|^2+4\int_0^t( f,\omega C+(1-\omega)\rho_s q(C)-(\omega \hat{C}+(1-\omega)\rho_s q(\hat{C})))\,dr
\notag
\\
& \qquad
- 2\omega\int_0^t \Bigg(\int_{\Gamma_{\text{in}}}   g^2(\bu\cdot \overrightarrow{n})ds\Bigg)\, dr-4(1-\omega)\rho_s\int_0^t \Bigg(\int_{\Gamma_{\text{in}}}   Q(g)(\bu\cdot \overrightarrow{n})ds\Bigg)\, dr+\mathcal{B}(t) .
\notag
\end{align}
\end{theorem}
\begin{proof}
Let ${\hat{C}}\in H^1(\Omega)$ such that ${{\hat{C}} \Big{|}}_{\Gamma_{\text{in}}}=g$. 
We test \eqref{variational1new} with $v=(\omega C+(1-\omega)\rho_s q(C))-(\omega \hat{C}+(1-\omega)\rho_s q(\hat{C}))\in H_{0,\Gamma_{\text{in}}}^1(\Omega).$
By using the divergence theorem and the boundary conditions, we get
\begin{align}
\label{s22neweq}
& 
( \bu\cdot \grad C,\omega C+(1-\omega)\rho_s q(C))
= \frac{\omega}{2}\int_{\Gamma_{\text{in}}}   g^2(\bu\cdot \overrightarrow{n})ds+  \frac{\omega}{2}\int_{\Gamma_{\text{out}}}   C^2(\bu\cdot \overrightarrow{n})ds
\\
&
\qquad
+(1-\omega)\rho_s \int_{\Gamma_{\text{in}}}   Q(g)(\bu\cdot \overrightarrow{n})ds+(1-\omega)\rho_s \int_{\Gamma_{\text{out}}}   Q(C)(\bu\cdot \overrightarrow{n})ds,
\notag
\end{align}
and
\begin{align}
& 
\omega(D\grad C, \grad C)+(1-\omega)\rho_s (q'(C)D\grad C, \grad C) 
\notag
\\
& 
=
\omega( D^{1/2} \grad C, D^{1/2}\grad C)+(1-\omega)\rho_s( q'(C)D^{1/2} \grad C, D^{1/2}\grad C)
\notag
\\
&
=
\omega \|D^{1/2}\grad C\|^2+(1-\omega)\rho_s\int_{\Omega}q'(C)(D^{1/2}\grad C)^2\ d\Omega.
\notag
\end{align}
Next, 
we move the terms involving ${\hat{C}}$ to the right-hand side
, and express them as follows
\begin{align}
( \bu\cdot \grad C,\omega{\hat{C}})
& 
=
\frac{3\omega}{8}(D^{1/2}\grad C,\frac{8}{3}D^{-1/2}{\hat{C}}\bu)
\notag
\\
& 
=
\frac{3\omega}{16}\|D^{1/2}\grad C\|^2+\frac{3\omega}{16}\bigg\|\frac{8}{3}D^{-1/2}{\hat{C}}\bu\bigg\|^2-\frac{3\omega}{16}\bigg\|D^{1/2}\grad C-\frac{8}{3}D^{-1/2}{\hat{C}}\bu\bigg\|^2,
\notag
\end{align}
and

\begin{align}
&
( \bu\cdot \grad C,(1-\omega)\rho_s q(\hat{C})) 
=
\frac{3\omega}{16}\|D^{1/2}\grad C\|^2
\notag
\\
& \qquad
+ \frac{3\omega}{16}\bigg\|\frac{8\rho_s q(\hat{C})(1-\omega)}{3\omega}D^{-1/2}\bu\bigg\|^2
- \frac{3\omega}{16}\bigg\|D^{1/2}\grad C-\frac{8\rho_s q(\hat{C})(1-\omega)}{3\omega}D^{-1/2}\bu\bigg\|^2.
\notag
\end{align}

Similarly,
\begin{align}
&
\omega(D\grad C, \grad \hat{C})
=
\frac{3\omega}{16}\|D^{1/2}\grad C\|^2+\frac{3\omega}{16}\bigg\|\frac{8}{3}D^{-1/2}\grad \hat{C}\bigg\|^2-\frac{3\omega}{16}\bigg\|D^{1/2}\grad C-\frac{8}{3}D^{-1/2}\grad \hat{C}\bigg\|^2,
\notag
\end{align}
and
\begin{align}
&
(D\grad C, (1-\omega)\rho_s q'(\hat{C})\grad \hat{C}) 
=
\frac{3\omega}{16}\|D^{1/2}\grad C\|^2
\notag
\\
& \,
+ \frac{3\omega}{16}\bigg\|\frac{8(1-\omega)\rho_s q'(\hat{C})}{3\omega}D^{-1/2}\grad \hat{C}\bigg\|^2-\frac{3\omega}{16}\bigg\|D^{1/2}\grad C-\frac{8(1-\omega)\rho_s q'(\hat{C})}{3\omega}D^{-1/2}\grad \hat{C}\bigg\|^2.
\notag
\end{align}
Finally,  we express the term involving the time derivative as
\begin{align}
& 
\bigg(\frac{\partial}{\partial t}(\omega C+(1-\omega)\rho_s q(C)),\omega \hat{C}+(1-\omega)\rho_s q(\hat{C}) \bigg)
\\
& 
=
\frac{\partial}{\partial t}( \omega C+(1-\omega)\rho_s q(C),\omega \hat{C}+(1-\omega)\rho_s q(\hat{C})).
\notag
\end{align}
Combining all, we get,
\begin{align*}
&
\frac{1}{2}\frac{\partial}{\partial t}\|\omega C+(1-\omega)\rho_s q(C)\|^2+\frac{\omega}{4} \|D^{1/2}\grad C\|^2+ \frac{\omega}{2}\int_{\Gamma_{\text{out}}}   C^2(\bu\cdot \overrightarrow{n})ds
\\
& \quad
+(1-\omega)\rho_s\int_{\Gamma_{\text{out}}}   Q(C)(\bu\cdot \overrightarrow{n})ds+(1-\omega)\rho_s\int_{\Omega}q'(C)(D^{1/2}\grad C)^2\ d\Omega
\\
& \quad
+\frac{3\omega}{16}\bigg\|D^{1/2}\grad C-\frac{8}{3}D^{-1/2}{\hat{C}}\bu\bigg\|^2
+\frac{3\omega}{16}\bigg\|D^{1/2}\grad C-\frac{8\rho_s q(\hat{C})(1-\omega)}{3\omega}D^{-1/2}\bu\bigg\|^2
\\
& \quad
+\frac{3\omega}{16}\bigg\|D^{1/2}\grad C-\frac{8}{3}D^{-1/2}\grad \hat{C}\bigg\|^2
+\frac{3\omega}{16}\bigg\|D^{1/2}\grad C-\frac{8(1-\omega)\rho_s q'(\hat{C})}{3\omega}D^{-1/2}\grad \hat{C}\bigg\|^2
\\
&
=( f,\omega C+(1-\omega)\rho_s q(C)-(\omega \hat{C}+(1-\omega)\rho_s q(\hat{C})))+\frac{3\omega}{16}\bigg\|\frac{8}{3}D^{-1/2}{\hat{C}}\bu\bigg\|^2
\\
& \quad
+\frac{3\omega}{16}\bigg\|\frac{8\rho_s q(\hat{C})(1-\omega)}{3\omega}D^{-1/2}\bu\bigg\|^2
+
\frac{3\omega}{16}\bigg\|\frac{8}{3}D^{-1/2}\grad \hat{C}\bigg\|^2
+\frac{3\omega}{16}\bigg\|\frac{8(1-\omega)\rho_s q'(\hat{C})}{3\omega}D^{-1/2}\grad \hat{C}\bigg\|^2
\\
& \quad
-\frac{\omega}{2}\int_{\Gamma_{\text{in}}}   g^2(\bu\cdot \overrightarrow{n})ds-(1-\omega)\rho_s\int_{\Gamma_{\text{in}}}   Q(g)(\bu\cdot \overrightarrow{n})ds
\\
& \quad
+\frac{\partial}{\partial t}( \omega C+(1-\omega)\rho_s q(C),\omega \hat{C}+(1-\omega)\rho_s q(\hat{C})).
\end{align*}
Integration on $[0,t]$ and the polarized identity 
yields \eqref{eq:stability-continuous}.
\end{proof}
A direct consequence of \Cref{thm:neth2neweq} is the following stability bound.
\begin{theorem}\label{neth2}
Assume that \ref{F1}-\ref{F6} are satisfied and the variational formulation given by \eqref{variational1new} has a solution $C\in L^{\infty}(0,T,L^2(\Omega))\cap L^{2}(0,T,H^1(\Omega))$ with $f\in L^2(0,T;L^2(\Omega))$. Let ${\hat{C}}$ be the continuous extension of the Dirichlet data $g$ inside the domain $\Omega$ and satisfies (\ref{construction}). The bounds on $\|\hat{C}\|^2$ and $\|\grad{\hat{C}}\|^2$ are given in \eqref{b1}. Let the antiderivative be $A(C)=\int_{0}^C s q'(s)ds$. Then we get the following bound:
\begin{align*}
&\|C(t)\|^2+\frac{4}{\omega}\int_{\Omega}(1-\omega)\rho_s A(C(t))d\Omega+\frac{\lambda}{\omega}\int_0^t \|\grad C(r)\|^2\ dr+\frac{2}{\omega}\int_0^t \bigg(\int_{\Gamma_{\text{out}}}   ((C)^2)(\bu\cdot \overrightarrow{n})ds\bigg)\ dr
\\&
\leq \frac{4}{\omega} \int_0^t
\frac{\|\bu\|_{\infty}^2}{\lambda} \|{\hat{C}}\|^2\ dr+\bigg(\frac{\lambda}{\omega}+\frac{4\beta_1^2}{\lambda\omega}\bigg)\int_0^t\|\grad {\hat{C}}\|^2\ dr   
-\frac{2}{\omega}\int_0^t \bigg(\int_{\Gamma_{\text{in}}}   ((g)^2)(\bu\cdot \overrightarrow{n})ds\bigg)\ dr
\\
& \qquad
+3\|C(0)\|^2 
 +\frac{8 K_{\text{PF}}^2}{\lambda\omega}\int_0^t \|f\|^2 \ dr+ \frac{16(\omega^2+(1-\omega)^2\rho_s^2K^2)}{\omega^2}\|{\hat{C}}\|^2
 \\
& \qquad
 +\frac{4}{\omega}\int_{\Omega}(1-\omega)\rho_s A(C(0))d\Omega.
\end{align*}
\end{theorem}
%
\begin{proof}
See \cite[Theorem 18]{siddiqua2024spurious}.
\end{proof}
\begin{remark}
We note that in the case of Langmuir's isotherm, the antiderivative is
$$A(C(t)) = \ln(1+C)+\frac{1}{1+C}+{\rm constant}.$$
\end{remark}
Next, we 
turn to the stability 
of the semidiscrete in space approximations 
in \eqref{variational2new},
and 
denote:
\begin{align*}
& 
Q_h(\alpha)=\int_{0}^{\alpha}\mathcal{P}( q(s)) ds
,
\\ 
& 
\mathcal{E}_h(t ) = \frac{3\omega}{4}\int_0^t\bigg\|D^{1/2}\grad C_h-\frac{8}{3}D^{-1/2}{\hat{C_h}}\bu\bigg\|^2\ dr+\frac{3\omega}{4}\int_0^t\bigg\|D^{1/2}\grad C_h-\frac{8}{3}D^{-1/2}\grad \hat{C_h}\bigg\|^2\ dr
\\
&
\qquad
+ \frac{3\omega}{4}\int_0^t\bigg\|D^{1/2}\grad C_h-\frac{8\rho_s \mathcal{P}(q(\hat{C_h}))(1-\omega)}{3\omega}D^{-1/2}\bu\bigg\|^2\ dr
\\
&
\qquad
+\frac{3\omega}{4}\int_0^t\bigg\|D^{1/2}\grad C_h-\frac{8(1-\omega)\rho_s \mathcal{P}^1(q'(\hat{C_h}))}{3\omega}D^{-1/2}\grad \hat{C_h}\bigg\|^2\ dr
\\
&
\qquad+ \|\omega C_h(t)+(1-\omega)\rho_s \mathcal{P}(q(C_h(t)))-2(\omega \hat{C_h}+(1-\omega)\rho_s \mathcal{P}(q(\hat{C_h})))\|^2,
\end{align*}
and
\begin{align*}
\mathcal{B}_h(t)
&
=\frac{3\omega}{4}\int_0^t\bigg\|\frac{8}{3}D^{-1/2}{\hat{C}}\bu\bigg\|^2\ dr
+\frac{3\omega}{4}\int_0^t\bigg\|\frac{8\rho_s\mathcal{P}( q(\hat{C_h}))(1-\omega)}{3\omega}D^{-1/2}\bu\bigg\|^2\ dr
\\
&
+  \frac{3\omega}{4}\int_0^t\bigg\|\frac{8}{3}D^{-1/2}\grad \hat{C_h}\bigg\|^2\ dr
+  \frac{3\omega}{4}\int_0^t\bigg\|\frac{8(1-\omega)\rho_s \mathcal{P}^1(q'(\hat{C}))}{3\omega}D^{-1/2}\grad \hat{C_h}\bigg\|^2\, dr
\\
&
+ \|\omega C(0)+(1-\omega)\rho_s \mathcal{P}(q(C(0)))-2(\omega \hat{C_h}+(1-\omega)\rho_s \mathcal{P}(q(\hat{C_h
})))\|^2.
\end{align*}
\begin{theorem}
\label{neth2neweqh}
Assume that \ref{F1}-\ref{F6} are satisfied, 
$C_h$ solves the semi-discrete in space Finite Element formulation with nonlinear adsorption 
\eqref{variational2new}, 
$f\in L^2(0,T;L^2(\Omega))$,
${\hat{C}}$ is 
the 
solution of (\ref{construction}), 
$Q_h(\alpha)
\geq 0$, and $\mathcal{P}^1(q'(C_h))\geq 0$. 
The following total mass balance 
relation holds:
\begin{align*}
&
\|\omega C_h(t)+(1-\omega)\rho_s \mathcal{P}(q(C_h(t)))\|^2+\omega\int_0^t \|D^{1/2}\grad C_h(r)\|^2\, dr
\\
&
\quad
+ 4(1-\omega)\rho_s\int_0^t\Bigg(\int_{\Omega}\mathcal{P}^1(q'(C_h(r)))(D^{1/2}\grad C_h(r))^2\, d\Omega\Bigg)\, dr
\\
&
\quad
+ 4 (1-\omega)\rho_s\int_0^t \bigg(\int_{\Gamma_{\text{out}}}   Q(C_h(r))(\bu\cdot \overrightarrow{n})ds\bigg)\, dr
+ 2\omega\int_0^t \bigg(\int_{\Gamma_{\text{out}}}   C_h^2(\bu\cdot \overrightarrow{n})ds\bigg)\, dr
+\mathcal{E}_h(t)
\\
&
= 
\|\omega C(0)+(1-\omega)\rho_s \mathcal{P}(q(C_h(0)))\|^2
+\mathcal{B}_h(t) 
\\
&
\quad
+4\int_0^t( f,\omega C_h+(1-\omega)\rho_s \mathcal{P}(q(C_h))-(\omega \hat{C_h}+(1-\omega)\rho_s \mathcal{P}(q(\hat{C_h}))))\,dr
\\
&
\quad
- 2\omega\int_0^t \bigg(\int_{\Gamma_{\text{in}}}   g_h^2(\bu\cdot \overrightarrow{n})ds\bigg)\, dr-4(1-\omega)\rho_s\int_0^t \bigg(\int_{\Gamma_{\text{in}}}   Q(g_h)(\bu\cdot \overrightarrow{n})ds\bigg)\, dr.
\end{align*}
\end{theorem}
\begin{proof}
Let ${\hat{C}_h}\in X^h(\Omega)$ such that ${{\hat{C}_h} \Big{|}}_{\Gamma_{\text{in}}}=g_h$ be the interpolant given in Lemma \ref{chath},
and
$\mathcal{P}$, 
$\mathcal{P}^1$ be the orthogonal projections with respect to the $L^2$ 
and $H^1$ inner products , respectively. 
Then we test \eqref{variational2new} with 
$v_h=(\omega C_h+(1-\omega)\rho_s \mathcal{P}(q(C_h)))-(\omega \hat{C_h}+(1-\omega)\rho_s \mathcal{P}(q(\hat{C_h})))\in X^h_{0,\Gamma_{\text{in}}}(\Omega)$
to obtain
\begin{align}
\label{eq1vc2neweqdis}
&
\bigg(\frac{\partial}{\partial t}(\omega C_h+(1-\omega)\rho_s q(C_h)),\omega C_h+(1-\omega)\rho_s \mathcal{P}(q(C_h))\bigg)+\omega(D\grad C_h, \grad C_h)
\\
&
\quad
+\bigg(\bu\cdot\grad C_h,\omega C_h+(1-\omega)\rho_s \mathcal{P}(q(C_h))\bigg)
   +(1-\omega)\rho_s (D\grad C_h, \mathcal{P}^1(q'(C_h)\grad C_h))
\notag
\\
&
=
( f,\omega C_h+(1-\omega)\rho_s \mathcal{P}(q(C_h))-(\omega \hat{C_h}+(1-\omega)\rho_s \mathcal{P}(q(\hat{C_h}))))
\notag\\
&
\quad
+\bigg(\frac{\partial}{\partial t}(\omega C_h+(1-\omega)\rho_s q(C_h)),(\omega \hat{C_h}+(1-\omega)\rho_s \mathcal{P}(q(\hat{C_h})))\bigg)+\omega(D\grad C_h, \grad \hat{C_h})
\notag
\\
&
\quad
+\bigg(\bu\cdot\grad C,(\omega \hat{C_h}+(1-\omega)\rho_s \mathcal{P}(q(\hat{C_h})))\bigg)
   +(1-\omega)\rho_s (D\grad C_h, \mathcal{P}^1(q'(\hat{C_h})\grad \hat{C_h})).
\notag   
\end{align}
We rewrite the first term in the left hand side as
\begin{align}
\label{s21neweq}
\bigg(\frac{\partial}{\partial t}(\omega C_h+(1-\omega)\rho_s q(C_h)),\omega C_h+(1-\omega)\rho_s \mathcal{P}(q(C_h))\bigg)
=
\frac{1}{2}\frac{\partial}{\partial t}\|\omega C_h+(1-\omega)\rho_s \mathcal{P}(q(C_h))\|^2.
\end{align}
Following the technique 
used in 
\Cref{thm:neth2neweq}
we get
\begin{align}
\label{s22neweqsemi}
&( \bu\cdot \grad C_h,\omega C_h+(1-\omega)\rho_s \mathcal{P}(q(C_h)))
\\&= \frac{\omega}{2}\int_{\Gamma_{\text{in}}}   g_h^2(\bu\cdot \overrightarrow{n})ds+  \frac{\omega}{2}\int_{\Gamma_{\text{out}}}   C_h^2(\bu\cdot \overrightarrow{n})ds
\notag
\\
&
\quad
+(1-\omega)\rho_s \int_{\Gamma_{\text{in}}}   Q_h(g_h)(\bu\cdot \overrightarrow{n})ds+(1-\omega)\rho_s \int_{\Gamma_{\text{out}}}   Q_h(C_h)(\bu\cdot \overrightarrow{n})ds
,
\notag
\end{align}
and
\begin{align}
&\omega(D\grad C_h, \grad C_h)+(1-\omega)\rho_s \mathcal{P}^1(q'(C_h)D\grad C_h, \grad C_h) 
\\
&
=\omega \|D^{1/2}\grad C_h\|^2+(1-\omega)\rho_s\int_{\Omega}\mathcal{P}^1(q'(C_h))(D^{1/2}\grad C_h)^2\ d\Omega.\notag
\end{align}
Also, the terms on the right-hand side 
write as
\begin{align}
\label{s24neweqsemi}
( \bu\cdot \grad C_h,\omega{\hat{C_h}})
&
=\frac{3\omega}{16}\|D^{1/2}\grad C_h\|^2+\frac{3\omega}{16}\bigg\|\frac{8}{3}D^{-1/2}{\hat{C_h}}\bu\bigg\|^2-\frac{3\omega}{16}\bigg\|D^{1/2}\grad C_h-\frac{8}{3}D^{-1/2}{\hat{C_h}}\bu\bigg\|^2,
    \end{align}
and
\begin{align}
( \bu\cdot \grad C_h,(1-\omega)\rho_s \mathcal{P}(q(\hat{C_h}))) 
\notag
&
=\frac{3\omega}{16}\|D^{1/2}\grad C_h\|^2+\frac{3\omega}{16}\bigg\|\frac{8\rho_s \mathcal{P}(q(\hat{C_h}))(1-\omega)}{3\omega}D^{-1/2}\bu\bigg\|^2
\\
&\qquad
-\frac{3\omega}{16}\bigg\|D^{1/2}\grad C_h-\frac{8\rho_s \mathcal{P}(q(\hat{C_h}))(1-\omega)}{3\omega}D^{-1/2}\bu\bigg\|^2.\notag
    \end{align}
Moreover,
\begin{align}
\label{s25neweqsemi}
&
\omega(D\grad C_h, \grad \hat{C_h})
\\
&
=\frac{3\omega}{16}\|D^{1/2}\grad C_h\|^2+\frac{3\omega}{16}\bigg\|\frac{8}{3}D^{-1/2}\grad \hat{C_h}\bigg\|^2-\frac{3\omega}{16}\bigg\|D^{1/2}\grad C_h-\frac{8}{3}D^{-1/2}\grad \hat{C_h}\bigg\|^2,
\notag
\\
&
(D\grad C_h, (1-\omega)\rho_s \mathcal{P}^1(q'(\hat{C_h}))\grad \hat{C_h}) 
\\
&
=\frac{3\omega}{16}\|D^{1/2}\grad C_h\|^2+\frac{3\omega}{16}\bigg\|\frac{8(1-\omega)\rho_s \mathcal{P}^1(q'(\hat{C_h}))}{3\omega}D^{-1/2}\grad \hat{C_h}\bigg\|^2
\notag
\\
&
\qquad
-\frac{3\omega}{16}\bigg\|D^{1/2}\grad C_h-\frac{8(1-\omega)\rho_s \mathcal{P}^1(q'(\hat{C_h}))}{3\omega}D^{-1/2}\grad \hat{C_h}\bigg\|^2
,
\notag
\end{align}
and 
\begin{align}
\label{s27neweqsemi}
&
\bigg(\frac{\partial}{\partial t}(\omega C_h+(1-\omega)\rho_s q(C_h)),\omega \hat{C_h}+(1-\omega)\rho_s \mathcal{P}(q(\hat{C_h}))\bigg)
\\
&
=\frac{\partial}{\partial t}( \omega C_h+(1-\omega)\rho_s q(C_h),\omega \hat{C_h}+(1-\omega)\rho_s \mathcal{P}(q(\hat{C_h})))
.
\notag
\end{align}
Combining \eqref{s21neweq
}-\eqref{s27neweqsemi} 
and
integrating both sides from $0$ to $t$, we obtain
\begin{align*}
&
\frac{1}{2}\|\omega C_h(t)+(1-\omega)\rho_s \mathcal{P}(q(C(t)))\|^2+\frac{\omega}{4}\int_0^t \|\grad C_h(r)\|^2\ dr+\frac{\omega}{2}\int_0^t \bigg(\int_{\Gamma_{\text{out}}}   C_h^2(\bu\cdot \overrightarrow{n})ds\bigg)\ dr
\\
&
\quad
+(1-\omega)\rho_s\int_0^t 
\int_{\Gamma_{\text{out}}}   Q_h(C_h(r))(\bu\cdot \overrightarrow{n})ds
\ dr+\frac{3\omega}{16}\int_0^t\bigg \|D^{1/2}\grad C_h(r)-\frac{8}{3}D^{-1/2}{\hat{C_h}}\bu\bigg\|^2\ dr
\\
&
\quad
+(1-\omega)\rho_s\int_0^t
\int_{\Omega}\mathcal{P}^1(q'(C_h(r)))(D^{1/2}\grad C_h(r))^2\ d\Omega
\ dr\notag
\\
&
\quad
+\frac{3\omega}{16}\int_0^t \bigg\|D^{1/2}\grad C_h(r)-\frac{8\rho_s \mathcal{P}(q(\hat{C_h}))(1-\omega)}{3\omega}D^{-1/2}\bu\bigg\|^2\ dr\notag
\\
&
\quad
+ \frac{3\omega}{16}\int_0^t \bigg\|D^{1/2}\grad C_h(r)-\frac{8}{3}D^{-1/2}\grad \hat{C_h}\bigg\|^2\ dr\notag
\\
&
\quad
+ \frac{3\omega}{16}\int_0^t\bigg\|D^{1/2}\grad C_h(r)-\frac{8(1-\omega)\rho_s \mathcal{P}^1(q'(\hat{C_h}))}{3\omega}D^{-1/2}\grad \hat{C_h}\bigg\|^2\ dr\notag
\\
&
=
\int_0^t( f,\omega C_h+(1-\omega)\rho_s \mathcal{P}(q(C_h))-(\omega \hat{C_h}+(1-\omega)\rho_s \mathcal{P}(q(\hat{C_h}))))\ dr+\frac{3\omega}{16}\int_0^t\bigg\|\frac{8}{3}D^{-1/2}{\hat{C_h}}\bu\bigg\|^2\, dr\notag
\\
&
\quad
+\frac{3\omega}{16}\int_0^t\bigg\|\frac{8\rho_s \mathcal{P}(q(\hat{C_h}))(1-\omega)}{3\omega}D^{-1/2}\bu\bigg\|^2\, dr
+
\frac{3\omega}{16}\int_0^t\bigg\|\frac{8}{3}D^{-1/2}\grad \hat{C_h}\bigg\|^2\, dr\notag
\\
&
\quad
+\frac{3\omega}{16}\int_0^t\bigg\|\frac{8(1-\omega)\rho_s \mathcal{P}^1(q'(\hat{C_h}))}{3\omega}D^{-1/2}\grad \hat{C_h}\bigg\|^2\, dr
+\frac{1}{2}\|\omega C_h(0)+(1-\omega)\rho_s \mathcal{P}(q(C_h(0)))\|^2\notag
\\
&
\quad
-\frac{\omega}{2}\int_0^t \bigg(\int_{\Gamma_{\text{in}}}   g_h^2(\bu\cdot \overrightarrow{n})ds\bigg)\, dr-(1-\omega)\rho_s\int_0^t \bigg(\int_{\Gamma_{\text{in}}}   Q_h(g_h)(\bu\cdot \overrightarrow{n})ds\bigg)\, dr\notag
\\
&
\quad
+( \omega C_h(t)+(1-\omega)\rho_s q(C_h(t)),\omega \hat{C_h}+(1-\omega)\rho_s \mathcal{P}(q(\hat{C_h})))\notag
\\
&
\quad
-( \omega C_h(0)+(1-\omega)\rho_s q(C_h(0)),\omega \hat{C_h}+(1-\omega)\rho_s \mathcal{P}(q(\hat{C_h}))).
\notag
\end{align*}
Writing the last two terms as 
\begin{align*}
&( \omega C_h(t)+(1-\omega)\rho_s q(C_h(t)),\omega \hat{C_h}+(1-\omega)\rho_s \mathcal{P}(q(\hat{C_h})))
\\&=\frac{1}{4}\| ( \omega C_h(t)+(1-\omega)\rho_s q(C_h(t))\|^2+\frac{1}{4} \|2(\omega \hat{C_h}+(1-\omega)\rho_s \mathcal{P}(q(\hat{C_h})))\|^2
\\
&
\quad
-\frac{1}{4} \|\omega C_h(t)+(1-\omega)\rho_s q(C_h(t))-2(\omega \hat{C_h}+(1-\omega)\rho_s \mathcal{P}(q(\hat{C_h})))\|^2,
\end{align*}
and
\begin{align*}
&
-( \omega C_h(0)+(1-\omega)\rho_s q(C_h0)),\omega \hat{C_h}+(1-\omega)\rho_s \mathcal{P}(q(\hat{C_h})))
\\&=-\frac{1}{4}\| ( \omega C_h(0)+(1-\omega)\rho_s q(C_h(0))\|^2-\frac{1}{4} \|2(\omega \hat{C_h}+(1-\omega)\rho_s \mathcal{P}(q(\hat{C_h})))\|^2
\\
&
\quad
+\frac{1}{4} \|\omega C_h(0)+(1-\omega)\rho_s q(C_h(0))-2(\omega \hat{C_h}+(1-\omega)\rho_s \mathcal{P}(q(\hat{C_h})))\|^2
,
\end{align*}
yields the claimed result.
\end{proof}
\subsubsection{Semi-discrete in space error estimate}
The following result 
gives an a priori error estimate 
of the semi-discrete in space approximation \eqref{timeintegfemsemi}
for the case of nonlinear adsorption.
\begin{theorem}\label{thnlerr}
Assume that \ref{F1}-\ref{F7} are satisfied, 
the variational formulation \eqref{timeintegv} with nonlinear adsorption 
has an exact solution $C\in H^1(0,T,H^{k+1}(\Omega))$, and $C_h$ solves the semi-discrete in space Finite Element formulation 
\eqref{timeintegfemsemi}. Then for all $1\leq r \leq k+1$ and each $T>0$ 
we have
\begin{align}
\label{ce11nl}
& 
{\omega}\int_0^T \|(C-C_h)\|^2dt
    + \|\int_0^TD^{1/2}\grad (C-C_h)dt'\|^2
\leq 
h^{2r-2}\int_0^T\|C\|_r^2dt
\\
& 
\quad
\times
{\rm exp}\bigg({T}+\frac{3T\|u\|_{\infty}^2\|D^{-1/2}\|_{\infty}^2}{\omega}\bigg)
    \bigg({4\|D^{1/2}\|_{\infty}^2}
+{3\omega}+4\|u\|_{\infty}^2\|D^{-1/2}\|_{\infty}^2+\frac{3(1-\omega)^2\rho_s^2 \kappa_2^2}{\omega}\bigg)K
^2
.
\notag
\end{align}    
\end{theorem}
\vspace{-0.2cm}
\begin{proof}
Let $v=v_h\in X_{0,\Gamma_{\text{in}}}^h\subset H_{0,\Gamma_{\text{in}}}^1(\Omega)$ in \eqref{timeintegv},
and then subtract \eqref{timeintegfemsemi} from \eqref{timeintegv} to obtain 
\begin{align}
\label{nleqerr0timeinteg}
&
0 
= 
(\omega (C-C_h),v_h)+((1-\omega)\rho_s( q(C)-q(C_h)),v_h)+\bigg(\int_0^tu\cdot\grad (C-C_h) dt',v_h\bigg)
\\
&
\qquad
-(\div\int_0^tD\grad (C-C_h)dt',v_h)
        ,\quad \forall v_h\in X_{0,\Gamma_{in}}^h(\Omega).\notag
\end{align}
Choosing $v_h=C_h-{\hat{C_h}}=C_h-C+C-{\hat{C_h}}\in X_{0,\Gamma_{\text{in}}}^h$ 
gives
\begin{align}
&
(\omega (C-C_h),C-C_h)+((1-\omega)\rho_s( q(C)-q(C_h)),C-C_h)
\notag
\\
&
\quad
+\bigg(\int_0^tu\cdot\grad (C-C_h) dt',C-C_h\bigg)
 +\bigg(\int_0^tD\grad (C-C_h)dt',\grad(C-C_h)\bigg)\notag
\\
&
=(\omega( C-C_h),C-{\hat{C_h}})+((1-\omega)\rho_s( q(C)-q(C_h)),C-{\hat{C_h}})
\notag
\\
&
\quad
+\bigg(\int_0^tu\cdot\grad (C-C_h) dt',C-{\hat{C_h}}\bigg)+\bigg(\int_0^tD\grad (C-C_h)dt',\grad(C-{\hat{C_h}})\bigg).\notag
\end{align}
The Cauchy-Schwarz inequality 
and integration on $[0,T]$ yields
\begin{align}
&
{\omega}\int_0^T\|C-C_h\|^2dt+2(1-\omega)\rho_s\int_0^T\int_{\Omega}\bigg( \int_0^1(q'(\theta C+(1-\theta)C_h))d\theta \bigg)(C-C_h)^2d\Omega\, dt
\notag
\\
&
\quad
+ \bigg\|\int_0^TD^{1/2}\grad (C-C_h)dt'\bigg\|^2\notag
\\
&
\leq 
\bigg({1}+\frac{3\|u\|_{\infty}^2\|D^{-1/2}\|_{\infty}^2}{\omega}\bigg)\int_0^T\|\int_0^tD^{1/2}\grad (C-C_h)dt'\|^2dt
    +{4\|D^{1/2}\|_{\infty}^2}\int_0^T\|\grad(C-{\hat{C_h}})\|^2dt\notag
\\
&
\quad
+\bigg({3\omega}+4\|u\|_{\infty}^2\|D^{-1/2}\|_{\infty}^2+\frac{3(1-\omega)^2\rho_s^2 \kappa_2^2}{\omega}\bigg)\int_0^T\|C-{\hat{C_h}}\|^2dt.\notag
\end{align}
Discarding the second term in the left hand side and using the 
Gronwall's inequality 
gives
\begin{align}
&
{\omega}\int_0^T \|(C-C_h)\|^2dt
    +\bigg\|\int_0^TD^{1/2}\grad (C-C_h)dt'\bigg\|^2
\notag    
\\
&
\leq 
\text{exp} \bigg( {T}+\frac{3T\|u\|_{\infty}^2\|D^{-1/2}\|_{\infty}^2}{\omega} \bigg)
\notag
\\
&
\quad
\times \bigg( {4\|D^{1/2}\|_{\infty}^2}+
+{3\omega}+4\|u\|_{\infty}^2\|D^{-1/2}\|_{\infty}^2+\frac{3(1-\omega)^2\rho_s^2 \kappa_2^2}{\omega} \bigg)
    \int_0^T\inf_{{\hat{C_h}}\in X^h\atop {\hat{C_h}}|_{\Gamma_{\text{in}}} = g_h}\|C-{\hat{C_h}}\|_1^2 dt
.
\notag
\end{align}
Let boundary term $g_h$ be the interpolant of $g$ in $X_{\Gamma_{\text{in}}}^h$.
Then 
Lemma \ref{interpolant} 
finally implies \eqref{ce11nl}.
\end{proof}
\subsubsection{Existence of the solution to the  fully discrete 
system}
In this subsection, we prove the solvability of the fully discrete system \eqref{nempcn}, approximating \eqref{Eq} for the nonlinear explicit isotherm.
By adding and subtracting $\hat{C}_h$ and multiplying by $2$, we 
rewrite 
\eqref{nempcn}
as follows:
Given ${C^{n}_h}-\hat{C}_h\in X_{0,\Gamma_{\text{in}}}^h$, find ${C^{n+1}_h}-\hat{C}_h\in X_{0,\Gamma_{\text{in}}}^h$ such that
\begin{align}
&
(D\nabla ({C^{n+1}_h}-\hat{C}_h),\nabla v_h )\notag
\\&
 =
 - 2\bigg( \bigg( \omega+(1-\omega)\rho_sq'(\frac{C_h^{n+1}-\hat{C}_h+C_h^n+\hat{C}_h}{2})\bigg)\frac{(C^{n+1}_h-\hat{C}_h)-(C^n_h-\hat{C}_h)}{\Delta t},v_h \bigg)
\notag
\\
\label{eu}
&
+ 2( f^{n+1/2},v_h)-(\bu \cdot \nabla ({C^{n+1}_h-\hat{C}_h+C^{n}_h}+\hat{C}_h ) , v_h)
+ (\nabla\cdot {(D\nabla(C^n_h+\hat{C}_h))},v_h),
\end{align}
for all $v_h\in X_{0,\Gamma_{\text{in}}}^h$.
To simplify the presentation, 
we drop the subscript $h$ throughout this section. 
We note that by the Lax-Milgram theorem \cite[corollary 5.8] {br}
we have that
$$\forall \ l\in X^{*},\ \text{there exists an unique solution}\ \Psi\in X_{0,\Gamma_{\text{in}}}\ \text{of}\ (D\grad\Psi,\grad v)=(l,v),\ \forall v\in X_{0,\Gamma_{\text{in}}},
$$
and therefore, the operator $T: X^{*}\rightarrow X_{0,\Gamma_{\text{in}}}$ defined by $T(l)=\Psi$ is a well-defined linear and continuous operator
:
$$\|T\|=\sup_{l\in X^{*}}\frac{\|T(l)\|_{X_{0,\Gamma_{\text{in}}}}}{\|l\|_{*}}=\sup_{l\in X^{*}}\frac{\|\grad \Psi\|}{\|l\|_{*}}\leq \frac{1}{\lambda},\ \text{since}\ \|\grad \Psi\|\leq \frac{1}{\lambda}\|l\|_{*}.$$
Next, we define the nonlinear operator 
$N: X_{0,\Gamma_{\text{in}}}\rightarrow X^{*}$ 
by
\begin{align*}
    N(\psi)&=2f^{n+\frac{1}{2}}-2\bigg(\omega+(1-\omega)\rho_s q'\bigg(\frac{\psi+C^n+\hat{C}}{2}\bigg)\bigg)\frac{\psi-(C^n-\hat{C})}{\Delta t}-\bu \cdot \nabla (\psi+C^n+\hat{C})  \\&+\div{(D\nabla(C^n+\hat{C}))},
\end{align*}
and the operator $\mathcal{F}: X_{0,\Gamma_{\text{in}}}\rightarrow X_{0,\Gamma_{\text{in}}}$ by $\mathcal{F}=T(N(\psi))$.
To prove the solvability of the problem \eqref{eu}, it suffices to show that 
$\mathcal{F}$ has a fixed point, i.e., 
there exists 
$\psi=\mathcal{F}(\psi) \in X_{0,\Gamma_{\text{in}}}$.
\begin{lemma}\label{eul1}
$N: X_{0,\Gamma_{\text{in}}}\rightarrow X^{*}$ is a bounded operator
\begin{align*}
\|N(\psi)\|_*
&
\leq 
\|\bu\|_{\infty}\|\grad{\psi}\| + 
\frac{2\omega+2(1-\omega)\kappa_2}{\Delta t}\|\psi\|
+ \frac{2\omega+2(1-\omega)\kappa_2}{\Delta t}\|C^n-\hat{C}\|+\|2f^{n+\frac{1}{2}}\|_*
\\
& 
\quad
+\|\div{(D\nabla (C^n+\hat{C}))}\|_* 
+ \|\bu\|_{\infty}\|\nabla({C^n+\hat{C}})\|.
\end{align*}
\end{lemma}
\begin{remark}
The proof of \Cref{eul1} follows from 
%
\ref{F6}, 
the triangle and Cauchy-Schwarz inequalities.
\end{remark}
\begin{lemma}\label{eul}
$N: X_{0,\Gamma_{\text{in}}}\rightarrow X^*$ is a continuous operator.
\end{lemma}
\begin{proof}
It suffices to show that $\|N(\psi_1)-N(\psi_2)\|_*\rightarrow 0$ if $\|\grad(\psi_1-\psi_2)\|
\rightarrow 0$ because $\|\psi\|_{X_{0,\Gamma_{\text{in}}}}$ is equivalent to $\|\grad \psi\|$. Here, 
\begin{align*}
   & \|N(\psi_1)-N(\psi_2)\|_*
    \\&\leq \frac{2\omega}{\Delta t}\|\psi_2-\psi_1\|_*+\|\bu\cdot \grad(\psi_2-\psi_1)\|_*
    \\&+\frac{2(1-\omega)\rho_s }{\Delta t}\|q'(\frac{\psi_2+C^n+\hat{C}}{2}){(\psi_2-(C^n-\hat{C})})-q'(\frac{\psi_1+C^n+\hat{C}}{2})({\psi_1-(C^n-\hat{C})})\|_*.
    \end{align*}
Notice, $\frac{2\omega}{\Delta t}\|\psi_2-\psi_1\|_*\leq \frac{2\omega K_{\text{PF}}}{\Delta t}\|\grad(\psi_2-\psi_1)\|$ and 
    $\|\bu\cdot \grad(\psi_2-\psi_1)\|_*\leq \|\bu\|_{\infty}\|\grad(\psi_2-\psi_1)\|.$\\
Next, $$\|q'(\frac{\psi_2+C^n+\hat{C}}{2})(\psi_2-\psi_1)\|_*\leq \kappa_2K_{\text{PF}}\|\grad(\psi_2-\psi_1)\|.$$
Using Lipschitz continuity of $q'$ we get, 
\begin{align*}
 &\|\bigg(q'(\frac{\psi_2+C^n+\hat{C}}{2})-q'(\frac{\psi_1+C^n+\hat{C}}{2})\bigg)(\psi_1-(C^n-\hat{C}))\|_*
 \\&\leq KK_{\text{PF}}\|\grad(\psi_2-\psi_1)\|\|\psi_1-(C^n-\hat{C})\|.   
\end{align*}
Hence, using Cauchy-Schwarz, Poincar\'e-Friedrichs inequalities and Lipschitz continuity of $q'$ we get,
\begin{align*}
& \|N(\psi_1)-N(\psi_2)\|_*
\\&\leq \bigg(\frac{2\omega K_{\text{PF}}}{\Delta t}+\|\bu\|_{\infty}+\frac{2(1-\omega)\rho_s \kappa_2K_{\text{PF}}}{\Delta t}+\frac{2(1-\omega)\rho_sKK_{\text{PF}} }{\Delta t}\|\psi_1-(C^n-\hat{C})\|\bigg) \\&\times\|\grad(\psi_2-\psi_1)\|
\end{align*}
which concludes the argument.
\end{proof}
\begin{lemma}\label{eul3}
 $\mathcal{F}: X_{0,\Gamma_{\text{in}}}\rightarrow X_{0,\Gamma_{\text{in}}}$ is a compact map.
\end{lemma}
\begin{proof}
Since $T: X^{*}\rightarrow X_{0,\Gamma_{\text{in}}}$ is a bounded linear operator, 
we only 
need to show that $N: X_{0,\Gamma_{\text{in}}}\rightarrow X^{*}$ is a compact map. 
By Lemmas \ref{eul1}-\ref{eul} we have that $N: X_{0,\Gamma_{\text{in}}}\rightarrow X^*$ is a bounded and continuous operator respectively, 
and the 
Rellich-Kondrachov 
theorem \cite[page 272]{evans2009partial} 
provides the compact embedding 
$I: X_{0,\Gamma_{\text{in}}}\hookrightarrow L^2$ defined by $I(\psi)=\psi$.
Therefore $N\circ I: X_{0,\Gamma_{\text{in}}}\hookrightarrow  L^2\rightarrow  X^{*}$ is compact.
%
\[
\begin{tikzcd}
  \psi \in X_{0,\Gamma_{in}} \arrow[r, hookrightarrow,"I"] \arrow[drr,"\mathcal{F}"'] &  L^2(\Omega) \arrow[r] &N(\psi)\in X^* \arrow[d,"T"] \\[2em]
  & & X_{0,\Gamma_{in}}
\end{tikzcd}  
\]
\end{proof}
%
%
\begin{theorem}
For any $v\in X_{0,\Gamma_{\text{in}}}$ and $f\in X^{*}$, 
there exists 
$\psi=C^{n+1}-\hat{C}\in X_{0,\Gamma_{\text{in}}}$ 
 solution to 
\cref{eu}. 
\end{theorem}
\begin{proof}
Consider $\psi_{\alpha}=\alpha \mathcal{F}(\psi_{\alpha})$ in $X_{0,\Gamma_{\text{in}}}$, 
$0\leq \alpha\leq 1$ 
defined by
\begin{align*}
\psi_{\alpha}
&
= T \bigg( 2\alpha f^{n+\frac{1}{2}}-2\alpha\bigg(\omega+(1-\omega)\rho_s q'(\frac{\psi+C^n+\hat{C}}{2})\bigg)\frac{\psi-(C^n-\hat{C})}{\Delta t}
\\
& \qquad 
- \alpha\bu\cdot \nabla ( \psi+C^n+\hat{C})+\alpha\div{(D\nabla(C^n+\hat{C}))} \bigg),
\end{align*}
which holds if and only if for all $v\in X_{0,\Gamma_{\text{in}}}$ $\psi_{\alpha}\in X_{0,\Gamma_{\text{in}}}$ satisfies
\begin{align}
(D\grad\psi_{\alpha},\grad{v})
& 
= 
-2\alpha \bigg( \bigg(\omega+(1-\omega)\rho_sq'(\frac{\psi_{\alpha}+C^n+\hat{C}}{2})\bigg)\frac{\psi_{\alpha}-(C^n-\hat{C})}{\Delta t},v \bigg)
\notag
\\
&
\qquad
+ 2( \alpha f^{n+1/2},v)-(\alpha \bu \cdot \grad({\psi_{\alpha}+C^{n}}+\hat{C}),v)-(\alpha{D\grad(C^n+\hat{C})},\grad v).
\notag
\end{align}
Then, by the Leray-Schauder fixed-point theorem 
\cite{MR433481,MR1814364}, we only need to prove 
an
a priori bound on $\|\grad \psi_{\alpha}\|$, independent of $\alpha$. 
This follows by s
etting $v=\psi_{\alpha}$ 
and using 
the
Cauchy-Schwarz and Poincar\'e-Friedrichs inequalities
\begin{align*}
\|\grad\psi_{\alpha}\|
&
\leq
\frac{2 K_{\text{PF}}}{\lambda}\|f^{n+1/2}\|+\frac{K_{\text{PF}}}{{\lambda}}\|\bu\cdot \nabla(C^n+\hat{C})\|+\frac{\beta_1}{\lambda}\|\nabla(C^n+\hat{C})\|
\\
&
\quad
+ \frac{2 K_{\text{PF}}(\omega+(1-\omega)\rho_s\kappa_2)}{{\lambda}\Delta t}\|C^n-\hat{C}\|
, \ \text{for $0\leq\alpha \leq 1$.}
\end{align*}
\end{proof}

\subsubsection{Stability of the solution to the  fully discrete 
system}
%
In this subsection, we derive an energy-like bound for \eqref{mp1}-\eqref{mp2},
the fully discrete version of the adsorption equation \eqref{Eq} for 
a
nonlinear, explicit isotherm, using the midpoint method for the time discretization.
We recall that at the continuous level, we proved that the solution $C>0$ is positive, and it is bounded by the initial and boundary conditions. 
Nevertheless, positivity at the discrete level and a discrete Maximum Principle are 
hard to 
obtain 
and usually hold under a CFL condition, i.e., the timestep has to be $\mathcal{O}(h^2)$ \cite{thomee2008existence
}. 
%
We 
use the following notations:
$Q_h(\alpha)=\int_{0}^{\alpha}\mathcal{P}( q(s)) ds$,
\begin{align*}
\mathcal{E}_h^n
=
& 
\frac{3\Delta t\omega}{4}\sum_{n=0}^{N}\bigg\|D^{1/2}\grad C_h^{n+1/2}-\frac{8}{3}D^{-1/2}{\hat{C_h}}\bu\bigg\|^2
\\
&
+\frac{3\Delta t\omega}{4}\sum_{n=0}^{N}\bigg\|D^{1/2}\grad C_h^{n+1/2}-\frac{8\rho_s \mathcal{P}(q(\hat{C_h}))(1-\omega)}{3\omega}D^{-1/2}\bu\bigg\|^2
\\
&
+\frac{3\Delta t\omega}{4}\sum_{n=0}^{N}\bigg\|D^{1/2}\grad C_h^{n+1/2}-\frac{8}{3}D^{-1/2}\grad \hat{C_h}\bigg\|^2
\\
&
+\frac{3\Delta t\omega}{4}\sum_{n=0}^{N}\bigg\|D^{1/2}\grad C_h^{n+1/2}-\frac{8(1-\omega)\rho_s \mathcal{P}^1(q'(\hat{C_h}))}{3\omega}D^{-1/2}\grad \hat{C_h}\bigg\|^2
\\
&
+ \|\omega C_h^{N+1}+(1-\omega)\rho_s q(C_h^{N+1})-2(\omega \hat{C_h}+(1-\omega)\rho_s \mathcal{P}(q(\hat{C_h})))\|^2,
\end{align*}
and
\begin{align*}
\mathcal{B}_h^n
=
& 
\frac{3N\Delta t\omega}{4}\bigg\|\frac{8}{3}D^{-1/2}{\hat{C_h}}\bu\bigg\|^2
+\frac{3N\Delta t\omega}{4}\bigg\|\frac{8\rho_s \mathcal{P}(q(\hat{C_h}))(1-\omega)}{3\omega}D^{-1/2}\bu\bigg\|^2
\\
&
+
\frac{3N\Delta t\omega}{4}\bigg\|\frac{8}{3}D^{-1/2}\grad \hat{C_h}\bigg\|^2
+\frac{3N\Delta t\omega}{4}\bigg\|\frac{8(1-\omega)\rho_s \mathcal{P}^1(q'(\hat{C_h}))}{3\omega}D^{-1/2}\grad \hat{C_h}\bigg\|^2
\\
&
+ \|\omega C_h^0+(1-\omega)\rho_s q(C_h^0)-2(\omega \hat{C_h}+(1-\omega)\rho_s \mathcal{P}(q(\hat{C_h})))\|^2.
\end{align*}
\begin{theorem}\label{thm:th4new}
Suppose the assumptions \ref{F1}-\ref{F7} 
hold, and
the fully discrete 
problem
\eqref{mp1}-\eqref{mp2} 
has a 
solution $\{C_h^n\}_{n=0}^N\in L^2(0,T;H^1(\Omega))$.
Then 
we have the following stability result
\begin{align*}
&
\|\omega C_h^{N+1}+(1-\omega)\rho_s \mathcal{P}(q(C_h^{N+1})\|^2+{\Delta t\omega}\sum_{n=0}^N \|\grad C_h^{n+1/2}\|^2
\\
& \quad
+{2\Delta t\omega}\sum_{n=0}^N\int_{\Gamma_{\text{out}}}   (C_h^{n+1/2})^2(\bu\cdot \overrightarrow{n})ds+ 4(1-\omega)\Delta t\rho_s\sum_{n=0}^N \int_{\Gamma_{\text{out}}}   Q_h(C_h^{n+1/2})(\bu\cdot \overrightarrow{n})ds
\\
&
\quad
+ 4(1-\omega)\Delta t\rho_s\sum_{n=0}^N\int_{\Omega}\mathcal{P}^1(q'(C_h^{n+1/2}))(D^{1/2}\grad C_h^{n+1/2})^2\ d\Omega+\frac{1}{4}\mathcal{E}_h^n
\\
&
=
\|\omega C_h^0+(1-\omega)\rho_s \mathcal{P}(q(C_h^0)\|^2+\mathcal{B}_h^n
\\
&
\quad
+ 4\Delta t\sum_{n=0}^N( f,\omega C_h^{n+1/2}+(1-\omega)\rho_s \mathcal{P}(q(C_h^{n+1/2}))-(\omega \hat{C_h}+(1-\omega)\rho_s\mathcal{P}( q(\hat{C_h}))))
\\
&
\quad
-{2N\Delta t\omega}\int_{\Gamma_{\text{in}}}   g_h^2(\bu\cdot \overrightarrow{n})ds-4(1-\omega)N\Delta t\rho_s \int_{\Gamma_{\text{in}}}   Q_h(g_h)(\bu\cdot \overrightarrow{n})ds.
\end{align*}
\end{theorem}
%
\begin{proof}
Let ${\hat{C_h}}\in X^h$  such that ${{\hat{C_h}} \Big{|}}_{\Gamma_{\text{in}}}=g_h$, 
and set
$$v_h=\omega C_h^{n+1/2}+(1-\omega)\mathcal{P}(q(C_h^{n+1/2}))-(\omega{\hat{C_h}+(1-\omega)\rho_s\mathcal{P}(q(\hat{C_h}})))\in  X_{0,\Gamma_{\text{in}}}^h(\Omega).
$$
Then, after plugging in $v_h$, we use the polarization identity to the first term of \eqref{mp1} and \eqref{mp2} and sum the resulting equations to  
yield
\begin{align}
\label{mp5n}
&
\frac{1}{2}(\|\omega C_h^{n+1} +(1-\omega)\rho_s\mathcal{P}(q(C_h^{n+1}))\|^2-\|\omega C_h^{n} +(1-\omega)\rho_s\mathcal{P}(q(C_h^{n}))\|^2
\\
&
\quad
- \|\omega C_h^{n+1/2} +(1-\omega)\rho_s\mathcal{P}((C_h^{n+1/2}))-(\omega C_h^{n+1} +(1-\omega)\rho_s\mathcal{P}(q(C_h^{n+1})))\|^2
\notag
\\
&
\quad
+\|\omega C_h^{n+1/2} +(1-\omega)\rho_s\mathcal{P}((C_h^{n+1/2}))-(\omega C_h^{n} +(1-\omega)\rho_s\mathcal{P}(q(C_h^{n})))\|^2)
\notag
\\
&
\quad
+ \Delta t( \bu\cdot \grad{C_h^{n+1/2}},\omega C_h^{n+1/2}+(1-\omega)\mathcal{P}(q(C_h^{n+1/2})))
\notag
\\
&
\quad
+ \Delta t( D\grad{C_h^{n+1/2}},\grad(\omega C_h^{n+1/2}+(1-\omega)\mathcal{P}(q(C_h^{n+1/2}))))
\notag
\\
&
=
\Delta t( f^{n+1/2},\omega C_h^{n+1/2}+(1-\omega)\mathcal{P}(q(C_h^{n+1/2}))-(\omega{\hat{C_h}+(1-\omega)\rho_s\mathcal{P}(q(\hat{C_h}}))))
\notag
\\
&
\quad
+ (\omega C_h^{n+1} +(1-\omega)\rho_sq(C_h^{n+1})- \omega C_h^{n}-(1-\omega)\rho_sq(C_h^{n}),(\omega{\hat{C_h}+(1-\omega)\rho_s\mathcal{P}(q(\hat{C_h}})))
\notag
\\
&
\quad
+\Delta t( \bu\cdot \grad{C_h^{n+1/2}},\omega{\hat{C_h}+(1-\omega)\rho_s\mathcal{P}(q(\hat{C_h}})))
\notag
+ \Delta t( D\grad{C_h^{n+1/2}},\grad((\omega{\hat{C_h}+(1-\omega)\rho_s\mathcal{P}(q(\hat{C_h}})))).
\notag
\end{align}
Since from 
\eqref{mp1}-\eqref{mp2} 
we also have
\begin{align*}
0 =     
&
-\|\omega C_h^{n+1/2} +(1-\omega)\rho_s\mathcal{P}((C_h^{n+1/2}))-(\omega C_h^{n+1} +(1-\omega)\rho_s\mathcal{P}(q(C_h^{n+1})))\|^2
\\
&
+\|\omega C_h^{n+1/2} +(1-\omega)\rho_s\mathcal{P}((C_h^{n+1/2}))-(\omega C_h^{n} +(1-\omega)\rho_s\mathcal{P}(q(C_h^{n})))\|^2
,
\end{align*}
then

\begin{align*}
\label{mp6n}
&
\frac{1}{2}(\|\omega C_h^{n+1} +(1-\omega)\rho_s\mathcal{P}(q(C_h^{n+1}))\|^2-\|\omega C_h^{n} +(1-\omega)\rho_s\mathcal{P}(q(C_h^{n}))\|^2)
\\
&
\quad
+ \Delta t( \bu\cdot \grad{C_h^{n+1/2}},\omega C_h^{n+1/2}+(1-\omega)\mathcal{P}(q(C_h^{n+1/2})))
\\
&
\quad
+ \Delta t( D\grad{C_h^{n+1/2}},\grad(\omega C_h^{n+1/2}+(1-\omega)\mathcal{P}(q(C_h^{n+1/2}))))
\\
&
=
\Delta t( f^{n+1/2},\omega C_h^{n+1/2}+(1-\omega)\mathcal{P}(q(C_h^{n+1/2}))-(\omega{\hat{C_h}+(1-\omega)\rho_s\mathcal{P}(q(\hat{C_h}}))))
\\
&
\quad
+(\omega C_h^{n+1} +(1-\omega)\rho_sq(C_h^{n+1})- \omega C_h^{n}-(1-\omega)\rho_sq(C_h^{n}),(\omega{\hat{C_h}+(1-\omega)\rho_s\mathcal{P}(q(\hat{C_h}})))
\\
&
\quad
+\Delta t( \bu\cdot \grad{C_h^{n+1/2}},(\omega{\hat{C_h}+(1-\omega)\rho_s\mathcal{P}(q(\hat{C_h}})))
\\
&
\quad
+ \Delta t( D\grad{C_h^{n+1/2}},\grad((\omega{\hat{C_h}+(1-\omega)\rho_s\mathcal{P}(q(\hat{C_h}})))).
\end{align*}
With a technique 
similar 
to the one used in the semidiscrete in space case in Theorem \ref{neth2neweqh} 
we get
\begin{align*}
&
\frac{1}{2}(\|\omega C_h^{n+1} +(1-\omega)\rho_s\mathcal{P}(q(C_h^{n+1}))\|^2-\|\omega C_h^{n} +(1-\omega)\rho_s\mathcal{P}(q(C_h^{n}))\|^2)+\frac{\omega \Delta t}{4}\|D^{1/2}\grad C_h^{n+1/2} \|^2 
\\
&
\quad
+(1-\omega)\Delta t\rho_s\int_{\Gamma_{\text{out}}}   Q_h(C_h^{n+1/2})(\bu\cdot \overrightarrow{n})ds+\Delta t(1-\omega)\rho_s\int_{\Omega}\mathcal{P}^1(q'(C_h^{n+1/2}))(D^{1/2}\grad C_h^{n+1/2})^2\, d\Omega
\\
&
\quad
+
\frac{\Delta t\omega}{2}\int_{\Gamma_{\text{out}}}   (C_h^{n+1/2})^2(\bu\cdot \overrightarrow{n})ds+\frac{3\Delta t\omega}{16}\bigg\|D^{1/2}\grad C_h^{n+1/2}-\frac{8}{3}D^{-1/2}{\hat{C_h}}\bu\bigg\|^2
\\
&
\quad
+\frac{3\Delta t\omega}{16}\bigg\|D^{1/2}\grad C_h^{n+1/2}-\frac{8\rho_s \mathcal{P}(q(\hat{C_h}))(1-\omega)}{3\omega}D^{-1/2}\bu\bigg\|^2
\\
&
\quad
+\frac{3\Delta t\omega}{16}\bigg\|D^{1/2}\grad C_h^{n+1/2}-\frac{8}{3}D^{-1/2}\grad \hat{C_h}\bigg\|^2
\\
&
\quad
+\frac{3\Delta t\omega}{16}\bigg\|D^{1/2}\grad C_h^{n+1/2}-\frac{8(1-\omega)\rho_s \mathcal{P}^1(q'(\hat{C_h}))}{3\omega}D^{-1/2}\grad \hat{C_h}\bigg\|^2
\\
&
=
\Delta t( f^{n+1/2},\omega C_h^{n+1/2}+(1-\omega)\rho_s \mathcal{P}(q(C_h^{n+1/2}))-(\omega \hat{C_h}+(1-\omega)\rho_s \mathcal{P}(q(\hat{C_h}))))
\\
&
\quad
-\frac{\Delta t\omega}{2}\int_{\Gamma_{\text{in}}}   g_h^2(\bu\cdot \overrightarrow{n})ds
+\frac{3\Delta t\omega}{16}\bigg\|\frac{8}{3}D^{-1/2}{\hat{C_h}}\bu\bigg\|^2
\\
&
\quad
+\frac{3\Delta t\omega}{16}\bigg\|\frac{8\rho_s \mathcal{P}(q(\hat{C_h}))(1-\omega)}{3\omega}D^{-1/2}\bu\bigg\|^2
+
\frac{3\Delta t\omega}{16}\bigg\|\frac{8}{3}D^{-1/2}\grad \hat{C_h}\bigg\|^2
\\
&
\quad
+\frac{3\Delta t\omega}{16}\bigg\|\frac{8(1-\omega)\rho_s \mathcal{P}^1(q'(\hat{C_h}))}{3\omega}D^{-1/2}\grad \hat{C_h}\bigg\|^2
-(1-\omega)\rho_s\Delta t\int_{\Gamma_{\text{in}}}   Q_h(g_h)(\bu\cdot \overrightarrow{n})ds
\\
&
\quad
+(\omega C_h^{n+1} +(1-\omega)\rho_sq(C_h^{n+1})- \omega C_h^{n}-(1-\omega)\rho_sq(C_h^{n}),(\omega{\hat{C_h}+(1-\omega)\rho_s\mathcal{P}(q(\hat{C_h}}))).
\end{align*}
Summation over $n=0$ to $n=N$
yields the conclusion.
\end{proof}


%
%
\subsection{Affine Isotherm}
 In the case of affine adsorption, i.e., $q(C)=K_1+K_2 C$ with $K_1,\ K_2\geq 0$, 
 we have that 
 $\frac{\partial q}{\partial t}=K_2 \frac{\partial C}{\partial t}$, 
 and let us denote
 $\bar{\omega}=(\omega+(1-\omega)\rho_s K_2)$. 
 The variational formulation 
 \eqref{variational1new}
 then 
 simplifies to
 :
\text{Find $C\in H^1(\Omega)$ such that ${C\Big{|}}_{\Gamma_{\text{in}}}=g$ and}
\begin{equation}\label{leq0vc}
   \bigg( \bar{\omega}\frac{\partial C}{\partial t},v \bigg) + ( \bu\cdot\grad C, v)+(D\grad C, \grad v)=( f,v),\text{ for all}\ \  v\in H_{0,\Gamma_{\text{in}}}^1(\Omega). 
\end{equation}
The semi-discrete in space Finite Element problem \eqref{variational2new} with affine adsorption is
:
Find $C_h\in X^h$ such that ${C_h  \Big{|}}_{\Gamma_{\text{in}}}=g_h$ and 
\begin{equation}
\label{leq1vc}
   \bigg( \bar{\omega}\frac{\partial C_h}{\partial t},v_h \bigg)
   + ( \bu\cdot\grad C_h, v_h)+(D\grad C_h, \grad v_h)=( f,v_h),\ \text{for all}\ \ v_h\in X_{0,\Gamma_{\text{in}}}^h(\Omega).
\end{equation}
For the fully discrete analysis, we 
use 
the refactorization of midpoint method \cite{burkardt2020refactorization} for time discretization
: 
Given $C^n_h\in X^h$, find $C^{n+1}_h\in X^h$ such that ${C_h^{n+1} \Big{|}}_{\Gamma_{\text{in}}}=g_h$ satisfying\\
Step 1: Backward Euler step at the half-integer time step $t_{n+1/2}$, for all $v_h\in X_{0,\Gamma_{\text{in}}}^h(\Omega)$,
\begin{align}\label{lmpc1}
\bigg( \bar{\omega}\frac{C^{n+1/2}_h-C^n_h}{\Delta t/2} , v_h \bigg)
+ (  u \cdot \grad{C^{n+1/2}_h},v_h)
+ (D\grad{C^{n+1/2}_h},\grad{v_h})
=
( f^{n+1/2},v_h).
\end{align}
Step 2: Forward Euler step at $t_{n+1}$, for all $v_h\in X_{0,\Gamma_{\text{in}}}^h(\Omega)$, 
\begin{align}\label{lmpc2}
\bigg( \bar{\omega}\frac{C^{n+1}_h-C^{n+1/2}_h}{\Delta t/2},v_h \bigg)
+ (  u \cdot \grad{C^{n+1/2}_h},v_h)
+
(D\grad{C^{n+1/2}_h},\grad{v_h})
=
( f^{n+1/2},v_h).
\end{align}
%
The following result 
gives an 
{\it a priori} error estimate for the case of affine adsorption 
$q(C) = K_1+K_2 C$
and semi-discrete in space approximations. 
\begin{theorem}\label{th3}
Assume that \ref{F1}-\ref{F6} are satisfied and the variational formulation with affine adsorption given by \eqref{leq0vc} has an exact solution $C\in H^1(0,T,H^{k+1}(\Omega))$ and $C_h$ solves the semi-discrete in space Finite Element formulation with affine adsorption given by \eqref{leq1vc}. Then for $1\leq r \leq k+1$ there exists a positive constant $K$ independent of $h$ such that:
\begin{align*}
\|C-C_h\|_{L^2(0,T;H^1( \Omega))}
\leq 
& 
\bigg(h^{r-1}\| C\|_{L^2(0,T;H^{r}(\Omega))}+h^{r-1}\Big\|\frac{\partial C}{\partial t}\Big\|_{L^2(0,T;H^r(\Omega))}+\|(C_h-{\hat{C_h}})(0)\|\bigg)
\\
& 
\times 
\bigg( 
\max\Big\{\bigg(2+\frac{8K_{\text{PF}}^2\|\bu\|_{\infty}^2+8\beta_1^2}{\lambda^2}\bigg) K_1^2,\frac{8K_{\text{PF}}^2\omega^2}{\lambda^2} K_1^2,\frac{4\omega}{\lambda}\Big\}
\bigg)^{1/2}
.
\end{align*}    
\end{theorem}
%
\begin{remark}
The argument is the same as in the nonlinear case, see Theorem \ref{thnlerr}.
\end{remark}
Next, we prove an energy type bound for \eqref{lmpc1}-\eqref{lmpc2},
the discrete version of the adsorption equation \eqref{Eq} for the 
affine isotherm, using the midpoint method for the time discretization. 
%
\begin{theorem}\label{th4}
Suppose the assumptions \ref{F1}-\ref{F7} 
hold, and
the fully discrete 
problem
\eqref{lmpc1}-\eqref{lmpc2}
has a smooth solution $\{C_h^n\}_{n=0}^N\in L^2(0,T;H^1(\Omega))$.
Then
\begin{align*}
&
\|C_h^{N}\|^2
+ \frac{2}{\bar{\omega}}\Delta t\sum_{n=0}^{N-1}\bigg(\int_{\Gamma_{\text{out}}}   ((C_h^{n+1/2})^2)(\bu\cdot \overrightarrow{n})ds\bigg)
+ \frac{\lambda}{\bar{\omega}}\Delta t \sum_{n=0}^{N-1}\|\grad{C_h^{n+1/2}}\|^2
\\
&
\leq 
\frac{4N\Delta t\|\bu\|_{\infty}^2+8\lambda\omega}{\bar{\omega}\lambda} \|{\hat{C_h}}\|^2
+ \frac{2N\Delta t}{\bar{\omega}}\bigg(\int_{\Gamma_{\text{in}}}   ((g_h)^2)(-\bu\cdot \overrightarrow{n})ds\bigg)
+ \frac{4N\Delta t \beta_1^2+N\Delta t \lambda^2}{\bar{\omega}\lambda}\|\grad {\hat{C_h}}\|^2
\\
&
\qquad
+ \frac{8K_{\text{PF}}^2}{\bar{\omega}\lambda}\Delta t\sum_{n=0}^{N-1}\|f^{n+1/2}\|^2
+ 3\|C_h^{0}\|^2.
\end{align*}
\end{theorem}
%
\begin{proof}
Let ${\hat{C_h}}\in X^h$ be such that ${{\hat{C_h}} \Big{|}}_{\Gamma_{\text{in}}}=g_h$ 
and take $v_h=C_h^{n+1/2}-{\hat{C_h}}\in  X_{0,\Gamma_{\text{in}}}^h(\Omega)$.
After plugging in $v_h$, we use the polarization identity to the first term of \eqref{lmpc1} and \eqref{lmpc2} and sum the resulting equations to
yield
\begin{align}
&
\frac{\bar{\omega}}{2}(\|C_h^{n+1}\|^2-\|C_h^{n}\|^2)+\Delta t ( \bu\cdot \grad{C_h^{n+1/2}},C_h^{n+1/2})+\Delta t( D\grad{C_h^{n+1/2}},\grad{C_h^{n+1/2}})
\notag
\\
&
= \Delta t ( f^{n+1/2},C_h^{n+1/2}-{\hat{C_h}})
+(\bar{\omega} ({C_h^{n+1} - C_h^{n}}),{\hat{C_h}})
+{\Delta t}( \bu\cdot \grad{C_h^{n+1/2}},{\hat{C_h}})
+ \Delta t( D\grad{C_h^{n+1/2}},\grad{{\hat{C_h}}}).
\notag
\end{align}
Using 
\begin{align}
( \bu\cdot \grad C_h^{n+1/2},C_h^{n+1/2})
= 
\frac{1}{2}\bigg(\int_{\Gamma_{\text{in}}}   ((g_h)^2)(\bu\cdot \overrightarrow{n})ds\bigg)
+  \frac{1}{2}\bigg(\int_{\Gamma_{\text{out}}}   ((C_h^{n+1/2})^2)(\bu\cdot \overrightarrow{n})ds\bigg)
,
\notag
\end{align}
and 
standard estimates 
we further obtain
\begin{align}
&
\frac{\bar{\omega}}{2}\|C_h^{n+1}\|^2-\frac{\bar{\omega}}{2}\|C_h^{n}\|^2+ \frac{\Delta t}{2}\bigg(\int_{\Gamma_{\text{out}}}   ((C_h^{n+1/2})^2)(\bu\cdot \overrightarrow{n})ds\bigg)+\frac{\Delta t\lambda}{4} \|\grad{C_h^{n+1/2}}\|^2
\notag
\\
&
\leq 
\frac{\|\bu\|_{\infty}^2\Delta t}{\lambda} \|{\hat{C_h}}\|^2-\frac{\Delta t}{2}\bigg(\int_{\Gamma_{\text{in}}}   ((g_h)^2)(\bu\cdot \overrightarrow{n})ds\bigg)
+ \frac{\Delta t \beta_1^2}{\lambda}\|\grad {\hat{C_h}}\|^2
\notag
\\
&
\quad
+ \frac{2K_{\text{PF}}^2\Delta t}{\lambda}\|f^{n+1/2}\|^2+\frac{\Delta t \lambda}{4}\|\grad {\hat{C_h}}\|^2
+(\bar{\omega} ({C_h^{n+1} - C_h^{n}}),{\hat{C_h}}),,
\notag
\end{align}
which by summation 
over $n=0$ to $n=N-1$ 
concludes the argument.
\end{proof}
\begin{remark}
We note that an intermediate result in the previous proof gives
\begin{align}
&
\frac{\bar{\omega}}{2}\|C_h^{n+1}\|^2-\frac{\bar{\omega}}{2}\|C_h^{n}\|^2+ \frac{\Delta t}{2}\bigg(\int_{\Gamma_{\text{out}}}   ((C_h^{n+1/2})^2)(\bu\cdot \overrightarrow{n})ds\bigg)+\Delta t( D\grad{C_h^{n+1/2}},\grad{C_h^{n+1/2}})
\notag
\\
&
=
-\frac{\Delta t}{2}\bigg(\int_{\Gamma_{\text{in}}}   ((g_h)^2)(\bu\cdot \overrightarrow{n})ds\bigg)+\Delta t ( f^{n+1/2},C_h^{n+1/2}-{\hat{C_h}})
\notag
\\
&
+(\bar{\omega} ({C_h^{n+1} - C_h^{n}}),{\hat{C_h}})
+{\Delta t}( \bu\cdot \grad{C_h^{n+1/2}},{\hat{C_h}})
+ \Delta t( D\grad{C_h^{n+1/2}},\grad{{\hat{C_h}}})
,
\notag
\end{align}
which when $f=0$ and $\hat{C_h}=0$ 
yields the following the mass balance type relation 
\begin{align}
&
\frac{\bar{\omega}}{2}(\|C_h^{n+1}\|^2-\|C_h^{n}\|^2)+ \frac{\Delta t}{2}\bigg(\int_{\Gamma_{\text{out}}}   ((C_h^{n+1/2})^2)(\bu\cdot \overrightarrow{n})ds\bigg)+\Delta t( D\grad{C_h^{n+1/2}},\grad{C_h^{n+1/2}})
\notag
\\
&
=
\frac{\Delta t}{2}\bigg(\int_{\Gamma_{\text{in}}}   ((g_h)^2)(-\bu\cdot \overrightarrow{n})ds\bigg)
,
\notag
\end{align}
where $\bu\cdot \overrightarrow{n}<0$ on $\Gamma_{\text{in}}$.
\end{remark}
Finally, we provide an {\it a priori} error estimate for the case of 
affine 
adsorption in the fully discrete case 
\eqref{lmpc1}-\eqref{lmpc2}.
%
\begin{theorem}\label{thm:th3time}
Suppose the assumptions \ref{F1}-\ref{F7} 
are satisfied, 
\eqref{lmpc1}-\eqref{lmpc2}
has a 
solution $\{C_h^n\}_{n=0}^N\in L^2(0,T;H^1(\Omega))$,
and 
\eqref{leq0vc} has an exact solution $C\in H^1(0,T,H^{k+1}(\Omega))$. 
Then for all $1\leq r \leq k+1$
:
\begin{align*}
&
\Delta t\sum_{n=0}^{N}\| C(t_{n+1/2})-C_h(t_{n+1/2})\|_1^2
\leq 
\max\Big\{\bigg(2+\frac{8K_{\text{PF}}^2\|\bu\|_{\infty}^2+8\beta_1^2}{\lambda^2}\bigg) K_2^2,\frac{8K_{\text{PF}}^2{\bar{\omega}}^2}{\lambda^2} K_2^2,\frac{
TK_{\text{PF}}^2}{ 3 
\lambda^2},\frac{2\bar{\omega}}{\lambda}\Big\}
\\
& 
\quad
\times
\bigg(h^{2r-2}\Delta t\sum_{n=0}^{N}\| C(t_{n+1/2})\|_r^2+h^{2r-2}\Big\|\frac{\partial C}{\partial t}\Big\|_{L^2(0,T;H^r(\Omega))}^2
+ (\Delta t)^4\|C_{ttt}\|_{L^{\infty}(0,T;L^{\infty})}^2
+ \|
\hat{C_h}-C_h^0\|^2\bigg).
\end{align*}
\end{theorem}
%
\begin{proof}
Let the approximate solution at time $t^{n+1/2}$ be $C_h^{n+1/2}$. Then by using the midpoint method, we get, the fully discrete variational formulation as follows:\\
Given $C^n_h\in X^h$, find $C^{n+1}_h\in X^h$ such that ${C^{n+1}_h  \Big{|}}_{\Gamma_{\text{in}}}=g_h$ and satisfying,
\begin{equation}\label{mp1time}
\begin{aligned}
\bigg(\bar{\omega}\frac{C^{n+1}_h-C^n_h}{\Delta t},v_h\bigg)+(  u \cdot \grad{C^{n+1/2}_h},v_h)+&(D\grad{C^{n+1/2}_h},\grad{v_h})=( f^{n+1/2},v_h), \ \forall v_h\in X_{0,\Gamma_{\text{in}}}^h(\Omega). 
\end{aligned}
\end{equation}
Let $C_t$ represent $\frac{\partial C}{\partial t}$. We write the following variational formulation for the exact solution $C(t)$.
\begin{equation}\label{eq0vctimemidf}
\begin{aligned}
  & \bigg(\bar{\omega}\frac{C(t_{n+1})- C(t_{n})}{\Delta t},v\bigg)+ ( \bu\cdot\grad C(t_{n+1/2}), v)+(D\grad C(t_{n+1/2}), \grad v)
   \\&=(f^{n+1/2},v)+(r^n,v),\  \forall \ v\in H_{0,\Gamma_{\text{in}}}^1(\Omega).
   \end{aligned}
\end{equation}
where time discretization error, $r^n=\frac{C(t_{n+1})-C(t_{n})}{\Delta t}-\frac{C_t(t_{n+1}
)+ C_t(t_{n}
)}{2}.$\\
Let $e^n=C(t_n)-C_h^n$ and $v=v_h\in X_{0,\Gamma_{\text{in}}}^h\subset H_{0,\Gamma_{\text{in}}}^1(\Omega)$ in \eqref{eq0vctimemidf} and then subtract \eqref{mp1time} from \eqref{eq0vctimemidf} to get 
\begin{align}
\bigg( \bar{\omega}\frac{e^{n+1}- e^{n}}{\Delta t},v_h \bigg)
+ ( \bu\cdot\grad e^{n+1/2}, v_h)+(D\grad e^{n+1/2}, \grad v_h)=(r^n,v_h),
   \ \  \forall v_h\in X_{0,\Gamma_{\text{in}}}^h(\Omega).
\notag
\end{align}
Consider 
${\hat{C_h}}\in X^h$ such that ${{\hat{C_h}} \Big{|}}_{\Gamma_{\text{in}}}=g_h$. 
Then $e^n = C(t_n)-C_h^n 
= \phi_h^n + \eta^n$, where 
$\phi_h^n={\hat{C_h}}-C_h^n$ and $\eta^n={\hat{C_h}}-C(t_n)$. 
Choosing $v_h=\phi_h^{n+1/2}\in X_{0,\Gamma_{\text{in}}}^h$, 
the above error equation becomes
\begin{align}
\label{ceqerr1time}
&
\bigg( \bar{\omega}\frac{\phi_h^{n+1}- \phi_h^{n}}{\Delta t},\phi_h^{n+1/2} \bigg)
+ ( \bu\cdot\grad \phi_h^{n+1/2}, \phi_h^{n+1/2})+(D\grad \phi_h^{n+1/2}, \grad \phi_h^{n+1/2})\notag 
\notag
\\
&
=
\bigg( \bar{\omega}\frac{\eta^{n+1}- \eta^{n}}{\Delta t} , \phi_h^{n+1/2} \bigg)
+ ( \bu\cdot\grad \eta^{n+1/2}, \phi_h^{n+1/2})+(D\grad \eta^{n+1/2}, \grad \phi_h^{n+1/2})+(r^n,\phi_h^{n+1/2}).
\notag
\end{align}
With a technique similar to one used in Theorem \ref{th3} 
and a Taylor expansion 
for the residual $r^n$
we obtain
\begin{align*}
\|\phi_h^{N}\|^2+\frac{\lambda}{\bar{\omega}}\sum_{n=0}^{N}\Delta t\|\grad \phi_h^{n+1/2}\|^2 
\leq 
& 
\bigg(\frac{4K_{\text{PF}}^2\|\bu\|_{\infty}^2
+ 4\beta_1^2}{\bar{\omega}\lambda}\bigg)\sum_{n=0}^{N}\Delta t\|\eta^{n+1/2}\|_1^2 
+ \frac{4K_{\text{PF}}^2\bar{\omega}}{\lambda}\int_{0}^{T}\|\eta_t\|^2 dt
\\
&
\quad 
+ 
\frac{TK_{\text{PF}}^2}{ 6 \bar{\omega}\lambda}
(\Delta t^2\|C_{ttt}\|_{L^{\infty}(0,T;L^{\infty})})^2+\|\phi_h^{0}\|^2.
\end{align*}
Since by the triangle inequality we have
\begin{align*}
&\sum_{n=0}^{N}\Delta t\| e^{n+1/2}\|_1^2
\leq 
\sum_{n=0}^{N} 2\Delta t \bigg(\| \phi_h^{n+1/2}\|_1^2+\|\eta^{n+1/2}\|_1^2\bigg)
,
\end{align*}
the errors satisfy
\begin{align*}
& 
\sum_{n=0}^{N}\Delta t\| e^{n+1/2}\|_1^2
\leq 
\bigg(2+\frac{8K_{\text{PF}}^2\|\bu\|_{\infty}^2+8\beta_1^2}{\lambda^2}\bigg) \Delta t \sum_{n=0}^{N} \inf_{{\hat{C_h}}\in X^h\atop {\hat{C_h}}|_{\Gamma_{\text{in}}} = g_h}\| C^{n+1/2}-{\hat{C_h}}\|_1^2 
\\
&
+ \frac{8 K_{\text{PF}}^2{\bar{\omega}}^2}{\lambda^2}\int_{0}^{T}
\inf_{{\hat{C_h}}\in X^h\atop {\hat{C_h}}|_{\Gamma_{\text{in}}} = g_h}
    \Big\|\frac{\partial (C-{\hat{C_h}})}{\partial t}\Big\|^2 dt
+ \frac{
TK_{\text{PF}}^2}{ 3 \lambda^2}(\Delta t^2\|C_{ttt}\|_{L^{\infty}(0,T;L^{\infty})})^2+\frac{2\bar{\omega}}{\lambda}\|\phi_h^{0}\|^2.
\end{align*}
Choosing $g_h$ 
the interpolant of $g$ in $X_{\Gamma_{\text{in}}}^h$, 
Lemma \ref{interpolant} 
concludes the argument.
\end{proof}
\section{Numerical Test}\label{sec:5}
In this section, we perform numerical tests to show that the midpoint method described in \Cref{subsec:32} gives a second-order convergence rate for the considered PDE model for the constant, affine, and nonlinear, explicit adsorptions. Since the results are similar, we only show the nonlinear, explicit adsorption case. In the following two subsections, we first check the convergence rate for the case of nonlinear, explicit isotherm in the first test, and in the second test, we plot the concentration profile. We also show the comparison of total mass after each test.
{
In order to reduce the computational cost, for the numerical results in this section we use a 2D spatial domain $\Omega=[0,1]^2\subset {\mathbb R}^2$, representing a cross-section through the center of the cylinder in Figure \ref{fig:domain}.}
{The numerical parameters are chosen to be representative of commonly used test cases \cite{thesis} in the numerical analysis literature and are intended to illustrate the theoretical properties of the scheme.}
\subsection{Convergence Results} 
For checking the order of convergence, we assume the following: $\bu=(1,1)$, $D=I$, $\Omega=[0,1]\times [0,1]$, $\omega=0.5$, $X^h$ = the space of continuous piecewise affine functions, the exact solution is $C(x,y,t)=t^2(x^3-\frac{3}{2}x^2+1)\cos{(\frac{\pi}{4} y)}$. The true solution determines the body force $f$, initial condition $C_0$, and boundary conditions.
The norms used in the table are defined as follows,
$$\| C\|_{\infty,0}:= \text{ess} \sup_{0<t<T} \| C(\cdot,t)\|_{L^2(\Og)}\ 
\text{and}
\ \| C\|_{0,0}
:=
\bigg(\int_0^T\| C(\cdot, t)\|_{L^2(\Og)}^2\ dt\bigg)^{1/2}.$$In this test problem, we use Langmuir's isotherm with $q_{max}=K_{eq}=1$ where $q(C)=\frac{q_{max}K_{eq}C}{1+K_{eq}C}=\frac{C}{1+C}$. We simplify the problem formulation to a single (nonlinear) transport equation in one unknown $C$ using
$$\frac{\partial q}{\partial t}=\frac{\partial q}{\partial C}\frac{\partial C}{\partial t}=\frac{1}{(1+C)^2}\frac{\partial C}{\partial t}.$$
While using Backward Euler discretization, we compute solutions by lagging the nonlinearity $q'(C_h^{n+1})$ as \cite{lagg}
$$q'(C_h^{n+1})\frac{C^{n+1}_h-C^{n}_h}{\Delta t}\approx q'(C_h^{n})\frac{C^{n+1}_h-C^{n}_h}{\Delta t}. $$
For the midpoint method, we use the standard (second order) linear extrapolation\cite{layton2020doubly} of $C^{n+1/2}_h$ while computing $q'(C^{n+1/2}_h)$ as
$$q'(C_h^{n+1/2})\frac{C^{n+1}_h-C^{n}_h}{\Delta t}\approx q'\bigg(\frac{3C_h^{n}-{C_h^{n-1}}}{2}\bigg)\frac{C^{n+1}_h-C^{n}_h}{\Delta t}. $$
\Cref{tab:first} reports BE temporal errors at fixed $h=1/128$ and the experimental order of convergence is $1$, which matches the first order precision of the scheme.

\begin{table}[ht]
\centering
\begin{tabular}{|l|l|l|l|l|l|}
\hline
$(h,\Delta t)\ \rightarrow$ & $(\sfrac{1}{128},\sfrac{1}{2})$ & $(\sfrac{1}{128},\sfrac{1}{4})$ & $(\sfrac{1}{128},\sfrac{1}{8})$ & $(\sfrac{1}{128},\sfrac{1}{16})$ & $(\sfrac{1}{128},\sfrac{1}{32})$\\ \hline
$\|C-C_h\|_{\infty,0} $       &0.0636074 &0.0374917 &0.0206665 &0.0108985 &0.00558454    \\ \hline
Rate                        &- &0.76262   &0.85928 & 0.92316  &0.96462  \\ \hline
$\|C-C_h\|_{0,0}$             &0.0522838 &0.0310798 &0.0169125 &0.00883222 &0.00451535       \\ \hline
Rate                        & -                             & 0.75039   &0.87789  &0.93724 & 0.96794                  \\ \hline
$\|\grad C-\grad C_h\|_{0,0}$ & 0.0847469 &0.0502647 &0.0273473 &0.0143409 &0.00746467  \\ \hline
Rate                        & -                            &0.75362 &0.87815 &0.93126   &0.94199\\ \hline
$\|C-C_h\|_{0,1}  $           &0.0995773 &0.0590973 &0.0321544 &0.0168424 &0.00872408     \\ \hline
Rate                        & -                             & 0.75272 &0.87808  &0.93292 &0.94902                            \\ \hline
\end{tabular}
\caption{Temporal convergence rates for the BE approximation with a Langmuir adsorption model to the non-steady-state problem.}
\label{tab:first}
\end{table}
\Cref{tab:second} shows midpoint temporal errors whose experimental order is 2 in all norms, consistent with the scheme’s second-order accuracy in the fully discrete analysis \Cref{thm:th3time}.
\begin{table}[ht]
\centering
\begin{tabular}{|l|l|l|l|l|l|}
\hline
$(h,\Delta t)\ \rightarrow$ & $(\sfrac{1}{128},\sfrac{1}{2})$ & $(\sfrac{1}{128},\sfrac{1}{4})$ & $(\sfrac{1}{128},\sfrac{1}{8})$ & $(\sfrac{1}{128},\sfrac{1}{16})$ & $(\sfrac{1}{128},\sfrac{1}{32})$\\ \hline
$\|C-C_h\|_{\infty,0}$        &0.0357416 &0.00951864 &0.00242801 &0.000611192 &0.000153313      \\ \hline
Rate                        &- &1.9088 &1.971 &1.9901  &1.9951    \\ \hline
$\|C-C_h\|_{0,0} $            &0.0307399 &0.00741601 &0.00181065 &0.00044712 &0.000111214           \\ \hline
Rate                        & -                             & 2.0514  &2.0341 &2.0178 &2.0073                  \\ \hline
$\|\grad C-\grad C_h\|_{0,0}$ & 0.744766 &0.191186 &0.0475431 &0.0117471 &0.00323681      \\ \hline
Rate                        & -                             &  1.9618  &2.0077  &2.0169 &1.8597                 \\ \hline
$\|C-C_h\|_{0,1} $            & 0.7454 &0.19133 &0.0475776 &0.0117556 &0.00323872            \\ \hline
Rate                        & -                             & 1.962 &2.0077  &2.0169 &1.8599                        \\ \hline
\end{tabular}
\caption{Temporal convergence rates for the midpoint approximation with a Langmuir adsorption model to the non-steady-state problem.}
\label{tab:second}
\end{table}
\Cref{fig:ratelangmuirq} compares BE and midpoint at $h=1/128$, the lines have log–log slopes $1$ and $2$, respectively, as predicted by our analysis.
\begin{figure}[H]
\centering
\includegraphics[width=1\columnwidth]{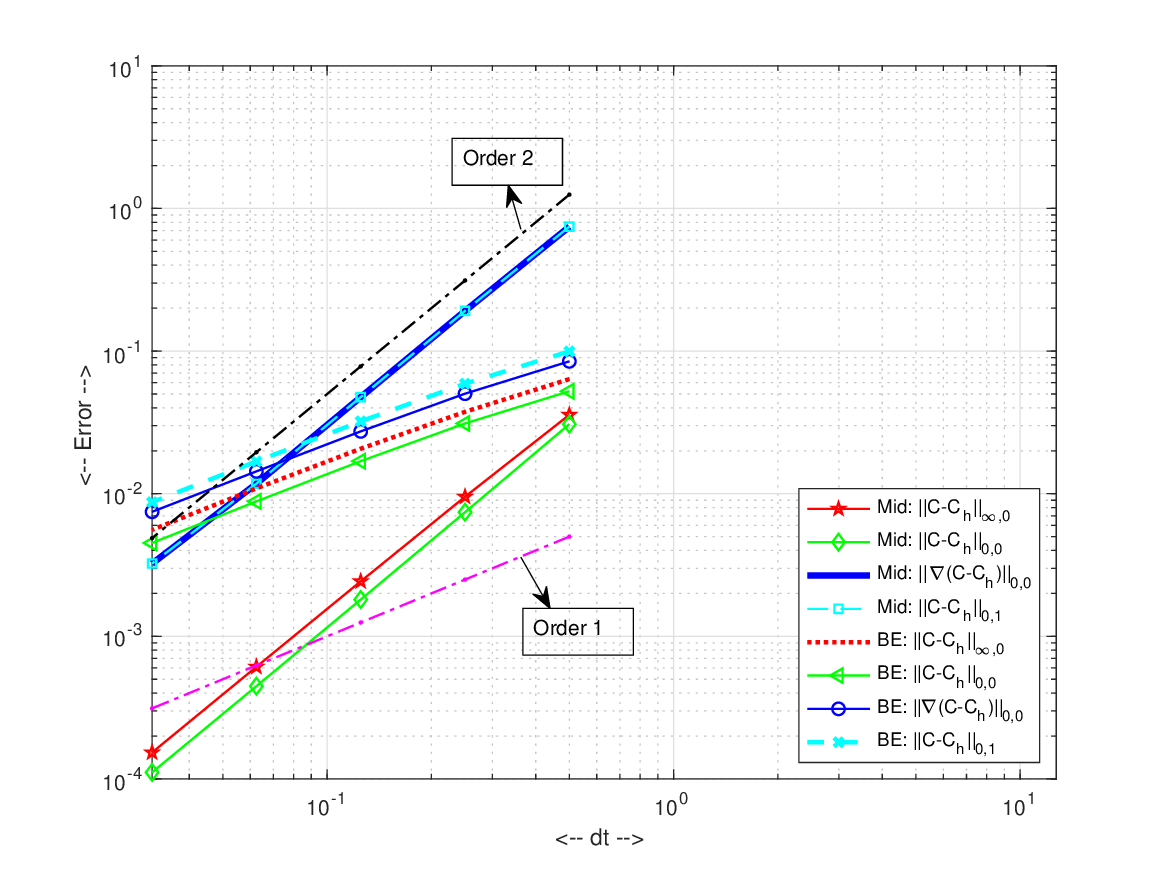}
\caption{Langmuir Isotherm: Temporal rate of convergence of BE and Midpoint, $T=1.0$, $h=1/128$. Notice that Midpoint is giving order 2 whereas BE is giving order 1.}
\label{fig:ratelangmuirq}
\end{figure}
\Cref{fig:totalmasslangmuirq} plots total-mass versus time. The midpoint method nearly preserves mass, whereas BE shows an upward deviation, which agrees with our mass-balance relation in \Cref{thm:neth2neweq} and the discrete stability bound for the midpoint scheme in \Cref{thm:th4new}.
\begin{figure}[H]
\centering
\includegraphics[width=1\columnwidth]{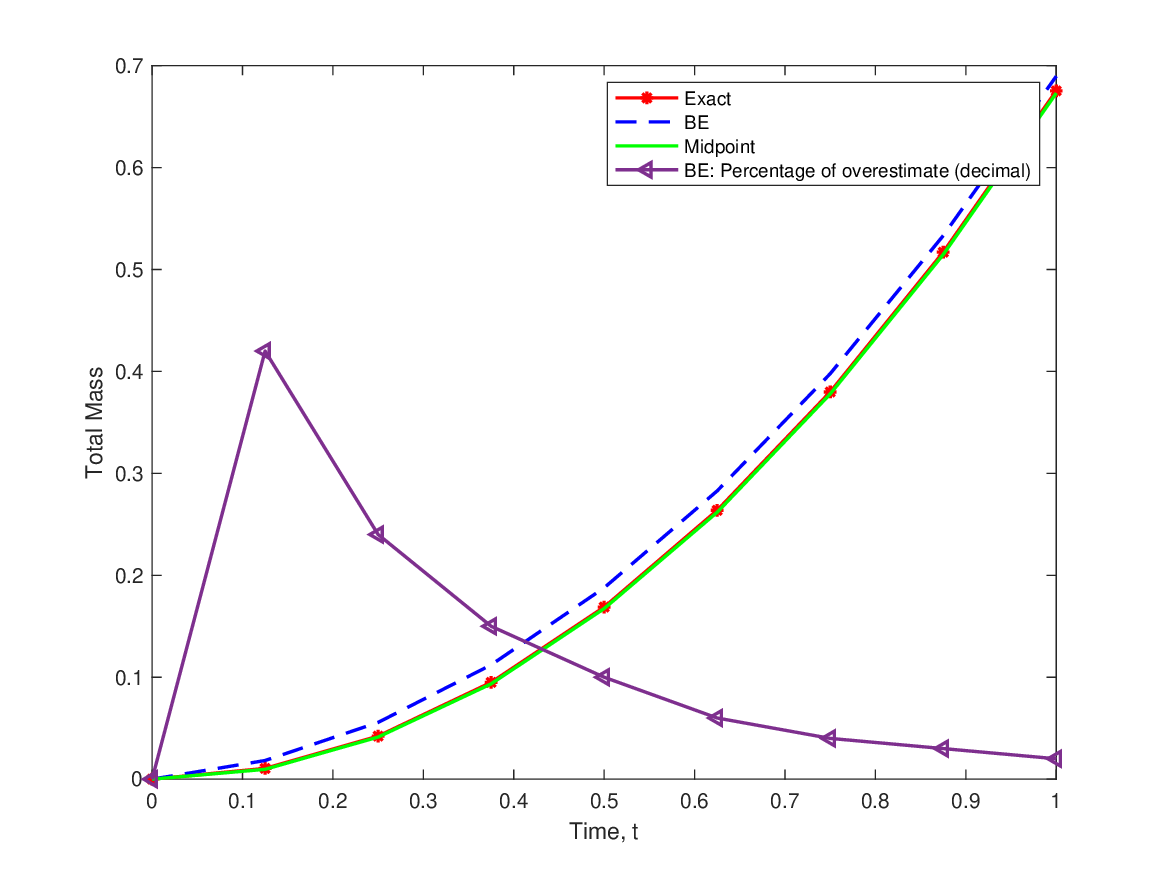}
\caption{Langmuir Isotherm: Comparison of total mass for exact solution, BE, Midpoint, $T=1.0$, $h=1/128$, $dt=1/8$. Notice that BE overestimates total mass rather than underestimates.}
\label{fig:totalmasslangmuirq}
\end{figure}
\subsection{Concentration Profiles}
For the plot of the concentration profile in each case, we consider the following:
$f=0$, $g=1$, $T=0.1 \ \text{and}\ T=3.0$, $h=1/128$, $dt=1/128$, $\bu=(0,2x(x-2))$, $D=I$, $\Omega=[0,2]\times [0,10]$, $\omega=0.5$, $X^h = \text{the space of continuous piecewise affine functions}.$
\begin{figure}[htbp]
\centering
\subfloat{\includegraphics[height=0.38\textheight,keepaspectratio]{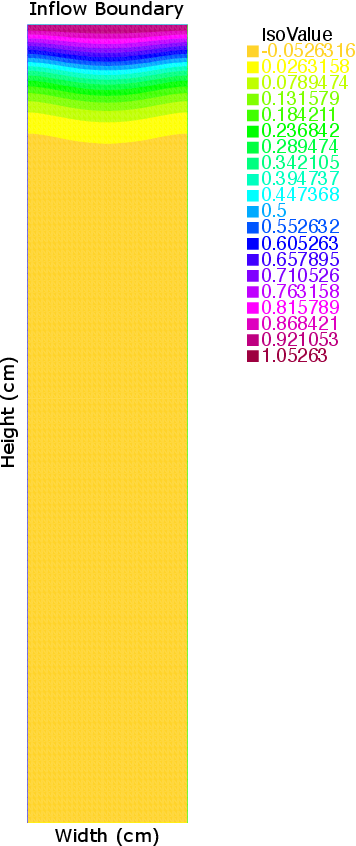}}
\hfill
\subfloat{\includegraphics[height=0.38\textheight,keepaspectratio]{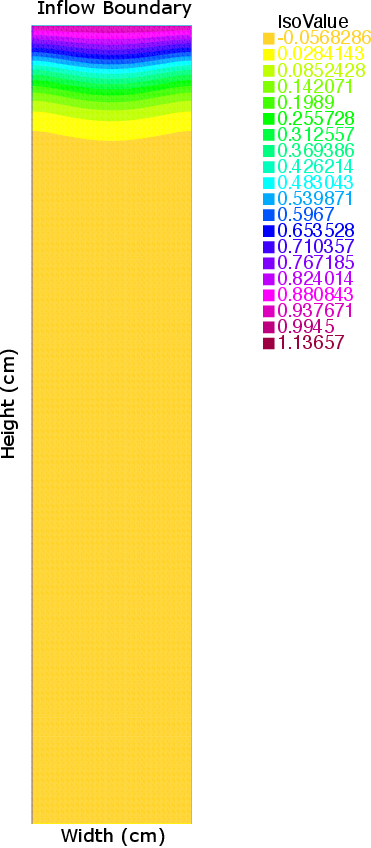}}
\caption{Langmuir isotherm: The evolution of the concentration in the membrane while using BE (Left) $\&$ Midpoint (Right), $\Omega=[0,2]\times [0,10]$, $f=0$, $g=1$, $T=0.1$, $h=1/128$, $dt=1/128$, $\bu=(0,2x(x-2))$, $D=I$.}
\label{fig:cplotlangmuir1}
\end{figure}
\begin{figure}[htbp]
\centering
\subfloat{\includegraphics[height=0.38\textheight,keepaspectratio]{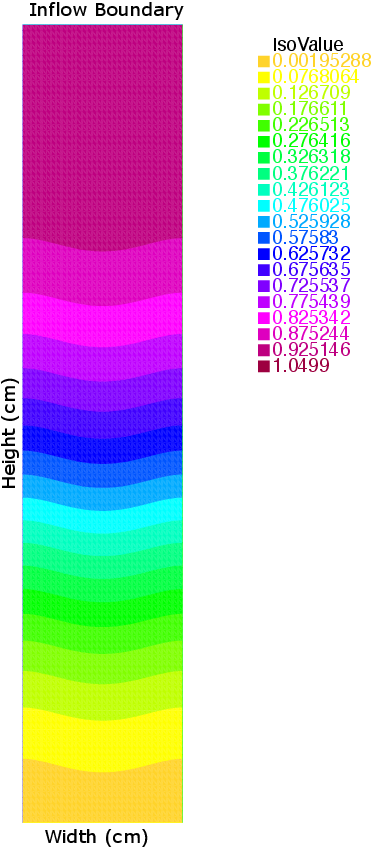}}
\hfill
\subfloat{\includegraphics[height=0.38\textheight,keepaspectratio]{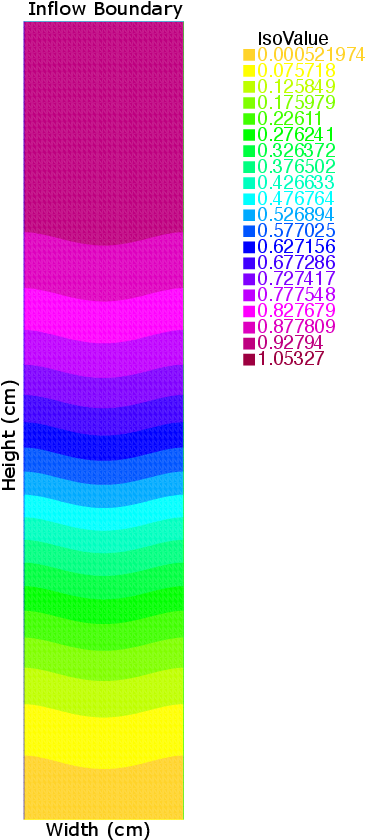}}
\caption{Langmuir isotherm: The evolution of the concentration in the membrane while using BE (Left) $\&$ Midpoint (Right), $\Omega=[0,2]\times [0,10]$, $f=0$, $g=1$, $T=3.0$, $h=1/128$, $dt=1/128$, $\bu=(0,2x(x-2))$, $D=I$.}
\label{fig:cplotlangmuir}
\end{figure}

\begin{figure}[H]
\centering
\includegraphics[width=1\columnwidth]{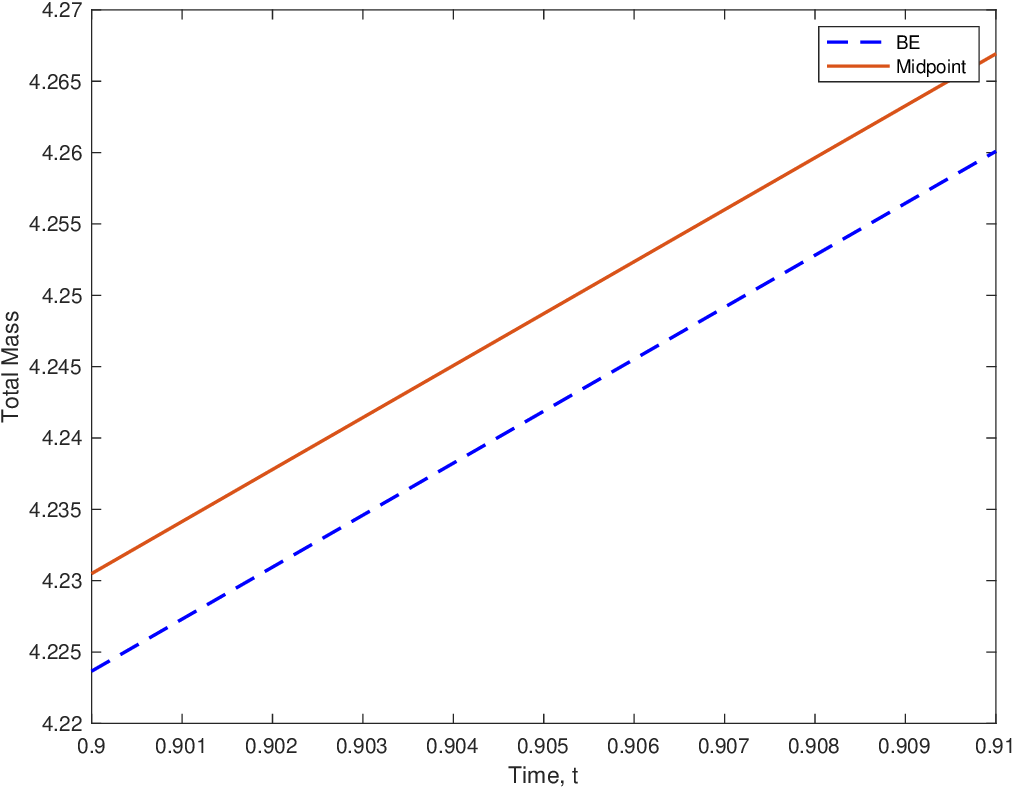}
\caption{Langmuir isotherm: Comparison of total mass, $\Omega=[0,2]\times [0,10]$, $f=0$, $g=1$, $T=3.0$, $h=1/128$, $dt=1/128$, $\bu=(0,2x(x-2))$, $D=I$.}
\label{totalclangmuir}
\end{figure}
In 
\Cref{fig:cplotlangmuir,fig:cplotlangmuir1}, 
the concentration front gradually advances through the height of the membrane (\textcolor{blue}{top to bottom}) over time as it evolves following the contour of the velocity profile.  Although we cannot visibly see the difference between two plots for BE and midpoint in \Cref{fig:cplotlangmuir}, we can see the significant difference in total mass evolution in \Cref{totalclangmuir}.
\section{Conclusion} 
We provided a detailed stability and error analysis of a simulation tool for modeling the adsorption process for the constant and affine adsorption cases. For the nonlinear, explicit adsorption, we proved stability analysis for the continuous case and semi-discrete case and the existence of a solution for the fully discrete case. The error analysis for this case is more involved and under some assumptions, we were able to show an error estimate for the semi-discrete case. But numerically, we showed that the midpoint method gives second-order convergence for all adsorption cases. The next most important step in developing this tool is coupling this reactive transport problem with porous media flow where velocity is approximated.

\bmhead{Acknowledgements}
We thank Professor William J. Layton, for his insightful idea and guidance throughout the research.
\section*{Declarations}
\begin{itemize}
\item Funding: Farjana Siddiqua 
is partially supported by the National Science Foundation under grant DMS-2110379. Catalin Trenchea is partially supported by the National Science Foundation under grant 2208220.
\item Conflict of interest: The authors declare no competing interests.
\item Ethics approval and consent to participate: Not Applicable.
\item Consent for publication: All authors approve the submission.
\item Data availability: Data sharing is not applicable to this article as no data sets were generated or analyzed during the current study.
\item Materials availability: Upon request.
\item Code availability: Upon request.
\item Author contribution: 
\begin{itemize}
\item Conceptualization: Farjana Siddiqua, Catalin Trenchea.
\item Analysis: Farjana Siddiqua.
\item Coding:  Farjana Siddiqua.
\item Writing – Original Draft:  Farjana Siddiqua
\item Writing – Review \& Editing: Farjana Siddiqua, Catalin Trenchea.
\end{itemize}

\end{itemize}
\bibliography{snref.bib}

@article {MR4967716,
    AUTHOR = {Correa, Maicon R. and Murad, Marcio A.},
     TITLE = {Fixed-stress sequential schemes for a black-oil model in
              poroelastic media},
   JOURNAL = {J. Comput. Phys.},
  FJOURNAL = {Journal of Computational Physics},
    VOLUME = {543},
      YEAR = {2025},
     PAGES = {Paper No. 114406, 39},
      ISSN = {0021-9991,1090-2716},
   MRCLASS = {76S (65M60 76M10 76S05)},
  MRNUMBER = {4967716},
       DOI = {10.1016/j.jcp.2025.114406},
       URL = {https://doi-org.pitt.idm.oclc.org/10.1016/j.jcp.2025.114406},
}

@ARTICLE{MuradCorrea2013,
  AUTHOR   = {Murad, Marcio Arab and Borges, Marcio and Obreg\'{o}n, Jesus A. and Correa, Maicon Ribeiro},
  TITLE    = {A new locally conservative numerical method for two-phase flow in heterogeneous
poroelastic media},
  YEAR     = {2013},
  JOURNAL  = {Computers and Geotechnics},
    VOLUME   = {48},
    NUMBER   = {},
    PAGES    = {192-207},
    DOI = {},
    URL = {} }

@article {MR3489128,
    AUTHOR = {List, Florian and Radu, Florin A.},
     TITLE = {A study on iterative methods for solving {R}ichards' equation},
   JOURNAL = {Comput. Geosci.},
  FJOURNAL = {Computational Geosciences. Modeling, Simulation and Data
              Analysis},
    VOLUME = {20},
      YEAR = {2016},
    NUMBER = {2},
     PAGES = {341--353},
      ISSN = {1420-0597,1573-1499},
   MRCLASS = {65M60 (65M22 76S05)},
  MRNUMBER = {3489128},
       DOI = {10.1007/s10596-016-9566-3},
       URL = {https://doi-org.pitt.idm.oclc.org/10.1007/s10596-016-9566-3},
}

@article {MR2114286,
    AUTHOR = {Radu, Florin and Pop, Iuliu Sorin and Knabner, Peter},
     TITLE = {Order of convergence estimates for an {E}uler implicit, mixed
              finite element discretization of {R}ichards' equation},
   JOURNAL = {SIAM J. Numer. Anal.},
  FJOURNAL = {SIAM Journal on Numerical Analysis},
    VOLUME = {42},
      YEAR = {2004},
    NUMBER = {4},
     PAGES = {1452--1478},
      ISSN = {0036-1429,1095-7170},
   MRCLASS = {65M12 (65M06 65M15 76M10 76S05)},
  MRNUMBER = {2114286},
MRREVIEWER = {Eun-Jae\ Park},
       DOI = {10.1137/S0036142902405229},
       URL = {https://doi-org.pitt.idm.oclc.org/10.1137/S0036142902405229},
}

@article{RaduPAMM,
author = {Radu, Florin A. and Pop, Iuliu S. and Attinger, Sabine and Knabner, Peter},
title = {Error estimates for an Euler implicit-mixed finite element scheme for reactive transport in saturated/unsaturated soil},
journal = {PAMM},
volume = {7},
number = {1},
pages = {1024705-1024706},
doi = {https://doi.org/10.1002/pamm.200700428},
url = {https://onlinelibrary.wiley.com/doi/abs/10.1002/pamm.200700428},
eprint = {https://onlinelibrary.wiley.com/doi/pdf/10.1002/pamm.200700428},
year = {2007}
}

@article {MR2652405,
    AUTHOR = {Radu, F. A. and Pop, I. S.},
     TITLE = {Newton method for reactive solute transport with equilibrium
              sorption in porous media},
   JOURNAL = {J. Comput. Appl. Math.},
  FJOURNAL = {Journal of Computational and Applied Mathematics},
    VOLUME = {234},
      YEAR = {2010},
    NUMBER = {7},
     PAGES = {2118--2127},
      ISSN = {0377-0427,1879-1778},
   MRCLASS = {65M60 (35K57 65M22 76S05 76V05)},
  MRNUMBER = {2652405},
MRREVIEWER = {Jason\ S.\ Howell},
       DOI = {10.1016/j.cam.2009.08.070},
       URL = {https://doi-org.pitt.idm.oclc.org/10.1016/j.cam.2009.08.070},
}

@BOOK{newref1,
  TITLE={Protein chromatography: process development and scale-up},
  author={Carta, G. and Jungbauer, A.},
  YEAR={2020},
  PUBLISHER={John Wiley \& Sons},
SERIES    = {},
    VOLUME    = {},
    EDITION   = {},
 ADDRESS   = {\empty} 
}

@BOOK{newref2,
  AUTHOR  = {Ghosh, R.},
  TITLE = {Principles of bioseparations engineering},
  PUBLISHER = {World Scientific Publishing company},
  YEAR = {2006},
 SERIES    = {},
    VOLUME    = {},
    EDITION   = {},
 ADDRESS   = {\empty} 
}

@BOOK{newref3,
  TITLE={Ion-exchange chromatography of proteins},
  AUTHOR={Yamamoto, S. and Nakanishi, K. and Matsuno, R.},
  YEAR={1988},
  PUBLISHER={CRC press},
    SERIES    = {},
    VOLUME    = {},
    EDITION   = {},
 ADDRESS   = {\empty} 
}

@BOOK{newref4,
  TITLE={High-performance membrane chromatography of proteins},
  AUTHOR={Tennikova, T. and Freitag, R.},
  BOOKTITLE={Analytical and preparative separation methods of biomacromolecules},
  PAGES={255--300},
  YEAR={2020},
  PUBLISHER={CRC Press},
ADDRESS   = {\empty} 
}

@article{newref5,
  title={A direct comparison between membrane adsorber and packed column chromatography performance},
  author={Boi, C. and Malavasi, A. and Carbonell, R. G. and Gilleskie, G.},
  journal={Journal of Chromatography A},
  volume={1612},
  pages={460629},
  year={2020},
  publisher={Elsevier}
}

@article{newref6,
  title={Downstream processing of monoclonal antibodies—application of platform approaches},
  author={Shukla, A. A. and Hubbard, B. and Tressel, T. and Guhan, S. and Low, D.},
  journal={Journal of Chromatography B},
  volume={848},
  number={1},
  pages={28--39},
  year={2007},
  publisher={Elsevier}
}

@article{newref7,
  title={Biocatalysis for pharmaceutical intermediates: the future is now},
  author={Pollard, D. J. and Woodley, J. M.},
  journal={TRENDS in Biotechnology},
  volume={25},
  number={2},
  pages={66--73},
  year={2007},
  publisher={Elsevier}
}

@article{wilson2020analysis,
  title         = {Analysis of a fully implicit {SUPG} scheme for a filtration and separation model},
  author        = {Wilson, A. B. and Jenkins, E. W.},
  journal       = {Computational and Applied Mathematics},
  volume        = {39},
  number        = {2},
  pages         = {1--19},
  year          = {2020},
  publisher     = {Springer}
}

@article{wilson2020numerical,
  title={Numerical simulation of chemical separations using multimodal adsorption isotherms},
  author={Wilson, A. B. and Jenkins, E. W. and Wang, J. and Husson, S. M.},
  journal={Results in Applied Mathematics},
  volume={7},
  pages={100122},
  year={2020},
  publisher={Elsevier}
}

@article{Lott2012,
TITLE    = {An accelerated {P}icard method for nonlinear systems related to variably
saturated flow },
  AUTHOR   = {Lott, P.A. and Walker, H.F. and Woodward, C.S. and Yang, U.M.},
  JOURNAL  = {Advances in Water Resources },
    VOLUME   = {38},
    NUMBER   = {},
    PAGES    = {92-101},
YEAR     = {2012},
}

@article{wilson2018numerical,
  title={Numerical Simulation of Solid Phase Adsorption Models Using Time-Integrated, Up-Winded Finite Element Strategies},
  author={Wilson, A. B. and Wang, J. and Jenkins, E. W. and Husson, S. M.},
  journal={Computing in Science \& Engineering},
  volume={22},
  number={3},
  pages={64--78},
  year={2018},
  publisher={IEEE}
}

@article{wang2015new,
  title={A new multimodal membrane adsorber for monoclonal antibody purifications},
  author={Wang, J. and Jenkins, E. W. and Robinson, J. R. and Wilson, A. and Husson, S. M.},
  journal={Journal of membrane science},
  volume={492},
  pages={137--146},
  year={2015},
  publisher={Elsevier}
}

@article{wang2017antibody,
  title={Antibody purification from {CHO} cell supernatant using new multimodal membranes},
  author={Wang, J. and Zhou, J. and Gowtham, Y. K. and Harcum, S. W. and Husson, S. M.},
  journal={Biotechnology Progress},
  volume={33},
  number={3},
  pages={658--665},
  year={2017},
  publisher={Wiley Online Library}
}

@article{brooks1982streamline,
  title={Streamline upwind/{P}etrov-{G}alerkin formulations for convection dominated flows with particular emphasis on the incompressible {N}avier-{S}tokes equations},
  author={Brooks, A. N. and Hughes, T.J.R.},
  journal={Computer methods in applied mechanics and engineering},
  volume={32},
  number={1-3},
  pages={199--259},
  year={1982},
  publisher={Elsevier}
}

@article{burman2010consistent,
  title={Consistent {SUPG}-method for transient transport problems: Stability and convergence},
  author={Burman, E.},
  journal={Computer Methods in Applied Mechanics and Engineering},
  volume={199},
  number={17-20},
  pages={1114--1123},
  year={2010},
  publisher={Elsevier}
}

@article{john2011error,
  title={Error analysis of the {SUPG} finite element discretization of evolutionary convection-diffusion-reaction equations},
  author={John, V. and Novo, J.},
  journal={SIAM journal on numerical analysis},
  volume={49},
  number={3},
  pages={1149--1176},
  year={2011},
  publisher={SIAM}
}

@article{hairer2006geometric,
  title={Geometric numerical integration},
  author={Hairer, E. and Hochbruck, M. and Iserles, A. and Lubich, C.},
  journal={Oberwolfach Reports},
  volume={3},
  number={1},
  pages={805--882},
  year={2006}
}

@article{sanz1992symplectic,
  title={Symplectic integrators for Hamiltonian problems: an overview},
  author={Sanz-Serna, J. M.},
  journal={Acta numerica},
  volume={1},
  pages={243--286},
  year={1992},
  publisher={Cambridge University Press}
}

@phdthesis{siddiqua2024spurious,
  title         = "Spurious Numerical Dissipation and Time Accuracy",
  author        = "Siddiqua, F.",
  year          = "2024",
  school        = "University of Pittsburgh"}

@article{nochetto1988approximation,
  title         = "Approximation of degenerate parabolic problems using numerical integration",
  author        = "Nochetto, R. H. and Verdi, C.",
  journal       = "SIAM Journal on Numerical Analysis",
  volume        = "25",
  number        = "4",
  pages         = "784--814",
  year          = "1988",
  publisher     = "SIAM"}

@article{arbogast1996nonlinear,
  title         = "A nonlinear mixed finite element method for a degenerate parabolic equation arising in flow in porous media",
  author        = "Arbogast, T.and Wheeler, M. F.",
  journal       = "SIAM Journal on Numerical Analysis",
  volume        = "33",
  number        = "4",
  pages         = "1669--1687",
  year          = "1996",
  publisher     = "SIAM"}

@article{woodward2000analysis,
  title         = "Analysis of expanded mixed finite element methods for a nonlinear parabolic equation modeling flow into variably saturated porous media",
  author        = "Woodward, C. S. and Dawson, C. N.",
  journal       = "SIAM Journal on Numerical Analysis",
  volume        = "37",
  number        = "3",
  pages         = "701--724",
  year          = "2000",
  publisher     = "SIAM"
}

@article{layton2002connection,
  title         = "A connection between subgrid scale eddy viscosity and mixed methods",
  author        = "Layton, W.",
  journal       = "Applied Mathematics and Computation",
  volume        = "133",
  number        = "1",
  pages         = "147--157",
  year          = "2002",
  publisher     = "Elsevier"
}

@article{quinteros2017therapeutic,
  title         = "Therapeutic use of monoclonal antibodies: general aspects and challenges for drug delivery",
  author        = "Quinteros, D. A. and Berm{\'u}dez, J. M. and Ravetti, S. and Cid, A. and Allemandi, D. A. and Palma, S. D.",
  booktitle     = "Nanostructures for Drug Delivery",
  pages         = "807--833",
  year          = "2017",
  publisher     = "Elsevier"
}

@article{oldwork1,
  title         = "Some upwinding techniques for finite element approximations of convection-diffusion equations",
  author        = "Bank, R. E. and B{\"u}rgler, J. F. and Fichtner, W. and Smith, R. K.",
  journal       = "Numerische Mathematik",
  volume        = "58",
  number        = "1",
  pages         = "185--202",
  year          = "1990",
  publisher     = "Springer"
}

@article{oldwork2,
  title         = "An improved cubic {P}etrov-{G}alerkin method for simulation of transient advection-diffusion processes in rectangularly decomposable domains",
  author        = "Bouloutas, E. T. and Celia, M. A.",
  journal       = "Computer Methods in Applied Mechanics and Engineering",
  volume        = "92",
  number        = "3",
  pages         = "289--308",
  year          = "1991",
  publisher     = "Elsevier"
}

@article{oldwork3,
  title         = "Comparison of some finite element methods for solving the diffusion-convection-reaction equation",
  author        = "Codina, R.",
  journal       = "Computer Methods in Applied Mechanics and Engineering",
  volume        = "156",
  number        = "1-4",
  pages         = "185--210",
  year          = "1998",
  publisher     = "Elsevier"
}

@article{oldwork4,
  title         = "Defect correction for convection-dominated flow",
  author        = "Heinrichs, W.",
  journal       = "SIAM Journal on Scientific Computing",
  volume        = "17",
  number        = "5",
  pages         = "1082--1091",
  year          = "1996",
  publisher     = "SIAM"
}

@article{oldwork5,
  title         = "Stabilized hp-finite element methods for first-order hyperbolic problems",
  author        = "Houston, P. and Schwab, C. and S{\"u}li, E.",
  journal       = "SIAM Journal on Numerical Analysis",
  volume        = "37",
  number        = "5",
  pages         = "1618--1643",
  year          = "2000",
  publisher     = "SIAM"
}

@article{oldwork6,
  title         = "Numerical solution of partial differential equations by the finite element method",
  author        = "Johnson, Claes",
  year          = "2012",
  publisher     = "Courier Corporation"
}

@article{adams2003sobolev,
  title         = "{S}obolev spaces",
  author        = "Adams, R. A. and Fournier, J. J.",
  year          = "2003",
  publisher     = "Elsevier"
}

@article{thomee2008existence,
  title         = "On the existence of maximum principles in parabolic finite element equations",
  author        = "Thom{\'e}e, V. and Wahlbin, L.",
  journal       = "Mathematics of Computation",
  volume        = "77",
  number        = "261",
  pages         = "11--19",
  year          = "2008"
}

@article{evans2009partial,
  title         = "Partial differential equations (graduate studies in mathematics, vol. 19)",
  author        = "Evans, L. C.",
  volume        = "67",
  year          = "2009",
  publisher     = "American Mathematics Society (1998) Text\# 2 (Required):"
}

@article{fix1983finite,
  title         = "On finite element approximations of problems having inhomogeneous essential boundary conditions",
  author        = "Fix, G. J. and Gunzburger, M. D. and Peterson, J. S.",
  journal       = "Computers \& Mathematics with Applications",
  volume        = "9",
  number        = "5",
  pages         = "687--700",
  year          = "1983",
  publisher     = "Elsevier"
}

@article{layton2020doubly,
  title         = "Doubly-adaptive artificial compression methods for incompressible flow",
  author        = "Layton, W. and McLaughlin, M.",
  journal       = "Journal of Numerical Mathematics",
  volume        = "28",
  number        = "3",
  pages         = "175--192",
  year          = "2020",
  publisher     = "De Gruyter"
}

@phdthesis{wang2016development,
  title         = "Development of a New Multimodal Membrane Adsorber and Its Application in {C}hromatographic {B}ioseparations",
  author        = "Wang, J.",
  year          = "2016",
  school        = "Clemson University"
}

@article{clarke2013bioprocess,
  title         = "Bioprocess engineering: an introductory engineering and life science approach",
  author        = "Clarke, K. G.",
  year          = "2013",
  publisher     = "Elsevier"
}

@article{coker2012biotherapeutics,
  title         = "Biotherapeutics outpace conventional therapies",
  author        = "Coker, V.",
  journal       = "BioPharm International",
  volume        = "25",
  number        = "3",
  pages         = "s20--s23",
  year          = "2012"
}

@article{leader2008protein,
  title         = "Protein therapeutics: a summary and pharmacological classification",
  author        = "Leader, B. and Baca, Q. J. and Golan, D. E.",
  journal       = "Nature Reviews Drug Discovery",
  volume        = "7",
  number        = "1",
  pages         = "21--39",
  year          = "2008",
  publisher     = "Nature Publishing Group"
}

@phdthesis{thesis,
  title         = "Modeling, Analysis, and Simulation of Adsorption in Functionalized Membranes",
  author        = "Wilson, A. B.",
  year          = "2016",
  school        = "Clemson University"
}

@article{m1,
  title         = "Formation of high-capacity protein-adsorbing membranes through simple adsorption of poly (acrylic acid)-containing films at low pH",
  author        = "Bhattacharjee, S. and Dong, J. and Ma, Y. and Hovde, S. and Geiger, J. H. and Baker, G. L. and Bruening, M. L.",
  journal       = "Langmuir",
  volume        = "28",
  number        = "17",
  pages         = "6885--6892",
  year          = "2012",
  publisher     = "ACS Publications"
}

@article{m2,
  title         = "Dramatic performance improvement of weak anion-exchange membranes for chromatographic bioseparations",
  author        = "Bhut, B. V. and Husson, S. M.",
  journal       = "Journal of Membrane Science",
  volume        = "337",
  number        = "1-2",
  pages         = "215--223",
  year          = "2009",
  publisher     = "Elsevier"
}

@article{m3,
  title         = "Preparation of high-capacity, weak anion-exchange membranes for protein separations using surface-initiated atom transfer radical polymerization",
  author        = "Bhut, B. V. and Wickramasinghe, S. R. and Husson, S. M.",
  journal       = "Journal of Membrane Science",
  volume        = "325",
  number        = "1",
  pages         = "176--183",
  year          = "2008",
  publisher     = "Elsevier"
}

@article{darcy,
  title         = "Quantitative Hydrogeology: Groundwater Hydrology for Engineers Academic Press",
  author        = "De Marsily, G.",
  journal       = "Inc., Orlando, Florida",
  year          = "1986"
}

@article{lang,
  title         = "Surface-initiated atom transfer radical polymerization: A new method for preparation of polymeric membrane adsorbers",
  author        = "Singh, N. and Wang, J. and Ulbricht, M. and Wickramasinghe, S. R. and Husson, S. M.",
  journal       = "Journal of Membrane Science",
  volume        = "309",
  number        = "1-2",
  pages         = "64--72",
  year          = "2008",
  publisher     = "Elsevier"
}

@article{bear1988dynamics,
  title         = "Dynamics of fluids in porous media",
  author        = "Bear, J.",
  year          = "1988",
  publisher     = "Courier Corporation"
}

@article{MR1814364,
  title         = "Elliptic partial differential equations of second order",
  author        = "Gilbarg, D. and Trudinger, N. S.",
  series        = "Classics in Mathematics",
  note          = "Reprint of the 1998 edition",
  publisher     = "Springer-Verlag",
  address       = "Berlin",
  year          = "2001",
  pages         = "xiv+517",
  isbn          = "3-540-41160-7"
}

@article{MR433481,
  title         = "Nonlinear functional analysis",
  author        = "Schwartz, J. T.",
  series        = "Notes on Mathematics and its Applications",
  note          = "Notes by H. Fattorini, R. Nirenberg and H. Porta, with an additional chapter by Hermann Karcher",
  publisher     = "Gordon and Breach Science Publishers, New York-London-Paris",
  year          = "1969",
  pages         = "vii+236"
}

@article{ervin2015limiting,
  title         = "On limiting behavior of contaminant transport models in coupled surface and groundwater flows",
  author        = "Ervin, V. J. and Kubacki, M. and Layton, W. J. and Moraiti, M. and Si, Z. and Trenchea, C.",
  journal       = "Axioms",
  volume        = "4",
  number        = "4",
  pages         = "518--529",
  year          = "2015",
  publisher     = "Multidisciplinary Digital Publishing Institute"
}

@article{chrispell2007fractional,
  title         = "A fractional step $\theta$-method for convection--diffusion problems",
  author        = "Chrispell, J. C. and Ervin, V. J. and Jenkins, E. W.",
  journal       = "Journal of Mathematical Analysis and Applications",
  volume        = "333",
  number        = "1",
  pages         = "204--218",
  year          = "2007",
  publisher     = "Elsevier"
}

@article{schneid2004priori,
  title         = "A priori error estimates for a mixed finite element discretization of the {R}ichards’ equation",
  author        = "Schneid, E. and Knabner, P. and Radu, F.",
  journal       = "Numerische Mathematik",
  volume        = "98",
  number        = "2",
  pages         = "353--370",
  year          = "2004",
  publisher     = "Springer"
}

@article{arbogast1995characteristics,
  title         = "A characteristics-mixed finite element method for advection-dominated transport problems",
  author        = "Arbogast, T. and Wheeler, M. F.",
  journal       = "SIAM Journal on Numerical Analysis",
  volume        = "32",
  number        = "2",
  pages         = "404--424",
  year          = "1995",
  publisher     = "SIAM"
}

@article{barrett1997finite,
  title         = "Finite element approximation of the transport of reactive solutes in porous media. Part II: Error estimates for equilibrium adsorption processes",
  author        = "Barrett, J. W. and Knabner, P.",
  journal       = "SIAM Journal on Numerical Analysis",
  volume        = "34",
  number        = "2",
  pages         = "455--479",
  year          = "1997",
  publisher     = "SIAM"
}

@article{dawson1998analysis,
  title         = "Analysis of an upwind-mixed finite element method for nonlinear contaminant transport equations",
  author        = "Dawson, C.",
  journal       = "SIAM Journal on Numerical Analysis",
  volume        = "35",
  number        = "5",
  pages         = "1709--1724",
  year          = "1998",
  publisher     = "SIAM"
}

@article{radu2010analysis,
  title         = "Analysis of an {E}uler implicit-mixed finite element scheme for reactive solute transport in porous media",
  author        = "Radu, F. A. and Pop, I. S. and Attinger, S.",
  journal       = "Numerical Methods for Partial Differential Equations: An International Journal",
  volume        = "26",
  number        = "2",
  pages         = "320--344",
  year          = "2010",
  publisher     = "Wiley Online Library"
}

@article{radu2011mixed,
  title         = "Mixed finite element discretization and {N}ewton iteration for a reactive contaminant transport model with nonequilibrium sorption: convergence analysis and error estimates",
  author        = "Radu, F. A. and Pop, I. S.",
  journal       = "Computational Geosciences",
  volume        = "15",
  number        = "3",
  pages         = "431--450",
  year          = "2011",
  publisher     = "Springer"
}

@article{burkardt2020refactorization,
  title         = "Refactorization of the midpoint rule",
  author        = "Burkardt, J. and Trenchea, C.",
  journal       = "Applied Mathematics Letters",
  volume        = "107",
  pages         = "106438",
  year          = "2020",
  publisher     = "Elsevier"
}

@article{brenner2008mathematical,
  title         = "The mathematical theory of finite element methods",
  author        = "Brenner, S. C. and Scott, L. R.",
  volume        = "3",
  year          = "2008",
  publisher     = "Springer"
}

@article{grisvard2011elliptic,
  title         = "Elliptic problems in nonsmooth domains",
  author        = "Grisvard, P.",
  year          = "2011",
  publisher     = "SIAM"
}

@article{Br,
  title         = "Functional analysis, {S}obolev spaces and partial differential equations",
  author        = "Brezis, H.",
  series        = "Universitext",
  publisher     = "Springer",
  address       = "New York",
  year          = "2011",
  pages         = "xiv+599",
  isbn          = "978-0-387-70913-0"
}

@article{gunzburger1992treating,
  title         = "Treating inhomogeneous essential boundary conditions in finite element methods and the calculation of boundary stresses",
  author        = "Gunzburger, M. D. and Hou, S. L.",
  journal       = "SIAM Journal on Numerical Analysis",
  volume        = "29",
  number        = "2",
  pages         = "390--424",
  year          = "1992",
  publisher     = "SIAM"
}

@article{dupont1980polynomial,
  title         = "Polynomial approximation of functions in {S}obolev spaces",
  author        = "Dupont, T. and Scott, R.",
  journal       = "Mathematics of Computation",
  volume        = "34",
  number        = "150",
  pages         = "441--463",
  year          = "1980"
}

@article{dauge2006elliptic,
  title         = "Elliptic boundary value problems on corner domains: smoothness and asymptotics of solutions",
  author        = "Dauge, M.",
  volume        = "1341",
  year          = "2006",
  publisher     = "Springer"
}

@article{scott1990finite,
  title         = "Finite element interpolation of nonsmooth functions satisfying boundary conditions",
  author        = "Scott, L. R. and Zhang, S.",
  journal       = "Mathematics of Computation",
  volume        = "54",
  number        = "190",
  pages         = "483--493",
  year          = "1990"
}

@article{lagg,
  title         = "Computational Fluid Mechanics and Heat Transfer",
  author        = "Dale A., John C. T. and Richard H. P.",
  publisher     = "Taylor \& Francis",
  year          = "2012",
  series        = "Series in Computational and Physical Processes in Mechanics and Thermal Sciences",
  edition       = "3rd"
}

@article{protter2012maximum,
  title         = "Maximum principles in differential equations",
  author        = "Protter, M. H. and Weinberger, H. F.",
  year          = "2012",
  publisher     = "Springer Science \\& Business Media"
}

@misc{website1,
  title         = "Biopharmaceutical Market Size to Hit US \$856.1 Billion by 2030",
  howpublished  = "\url{https://www.globenewswire.com/news-release/2021/12/22/2357003/0/en/Biopharmaceutical-Market-Size-to-Hit-US-856-1-Bn-by-2030.html}",
  note          = "Precedence Research, Accessed: 2022-08-01"
}

@misc{website2,
  title         = "Anti-SARS-CoV-2 Monoclonal Antibodies",
  howpublished  = "\url{https://www.covid19treatmentguidelines.nih.gov/therapies/anti-sars-cov-2-antibody-products/anti-sars-cov-2-monoclonal-antibodies/}",
  note          = "Accessed: 2022-08-18"
}

@misc{website3,
  title         = "COVID-19 Treatments and Medications",
  howpublished  = "\url{https://www.cdc.gov/coronavirus/2019-ncov/your-health/treatments-for-severe-illness.html}",
  note          = "Accessed: 2022-10-19"
}

@article{list2016study,
  title={A study on iterative methods for solving {R}ichards’ equation},
  author={List, F. and Radu, F. A.},
  journal={Computational Geosciences},
  volume={20},
  number={2},
  pages={341--353},
  year={2016},
  publisher={Springer}
}

@article{mitra2024posteriori,
  title={A posteriori error estimates for the {R}ichards equation},
  author={Mitra, K. and Vohral{\'\i}k, M.},
  journal={Mathematics of Computation},
  volume={93},
  number={347},
  pages={1053--1096},
  year={2024}
}

@article{congreve2025error,
  title={Error analysis for local discontinuous {G}alerkin semidiscretization of {R}ichards’ equation},
  author={Congreve, S. and Dolej{\v{s}}{\'\i}, V. and Saki{\'c}, S.},
  journal={IMA Journal of Numerical Analysis},
  volume={45},
  number={1},
  pages={580--630},
  year={2025},
  publisher={Oxford University Press}
}

@article{difonzo2024convergence,
  title={Convergence analysis of a spectral numerical method for a peridynamic formulation of {R}ichards’ equation},
  author={Difonzo, F. V. and Pellegrino, S. F.},
  journal={Mathematics and Computers in Simulation},
  volume={223},
  pages={219--228},
  year={2024},
  publisher={Elsevier}
}

@misc{website4,
  title         = "Coronavirus (COVID-19) Update: FDA Authorizes New Monoclonal Antibody for Treatment of COVID-19 that Retains Activity Against Omicron Variant",
  howpublished  = "\url{https://www.fda.gov/news-events/press-announcements/
}",
  note          = "Accessed: 2022-02-11"
}

@article {MR4766725,
    AUTHOR = {Jones, G. S. and Trenchea, C.},
     TITLE = {Discrete energy balance equation via a symplectic second-order
              method for two-phase flow in porous media},
   JOURNAL = {Appl. Math. Comput.},
  FJOURNAL = {Applied Mathematics and Computation},
    VOLUME = {480},
      YEAR = {2024},
     PAGES = {Paper No. 128909},
      ISSN = {0096-3003},
   MRCLASS = {Prelim},
  MRNUMBER = {4766725},
       DOI = {10.1016/j.amc.2024.128909},
       URL = {https://doi-org.pitt.idm.oclc.org/10.1016/j.amc.2024.128909},
}

\end{document}